\documentclass{amsart}%
\usepackage{amsfonts}
\usepackage{amsmath}
\usepackage{amssymb}
\usepackage{graphicx}%
\newtheorem{theorem}{Theorem}
\theoremstyle{plain}

\newtheorem{definition}{Definition}

\newtheorem{lemma}{Lemma}
\newtheorem{notation}{Notation}

\newtheorem{proposition}{Proposition}
\newtheorem{remark}{Remark}

\numberwithin{equation}{section}
\begin{document}
\title[Local zeta functions and fundamental solutions]{Local Zeta Functions, pseudodifferential operators, and Sobolev-type spaces
over non-Archimedean Local Fields}
\author{W. A. Z\'{u}\~{n}iga-Galindo}
\address{Centro de Investigaci\'{o}n y de Estudios Avanzados del Instituto
Polit\'{e}cnico Nacional\\
Departamento de Matem\'{a}ticas, Unidad Quer\'{e}taro\\
Libramiento Norponiente \#2000, Fracc. Real de Juriquilla. Santiago de
Quer\'{e}taro, Qro. 76230\\
M\'{e}xico.}
\email{wazuniga@math.cinvestav.edu.mx}
\thanks{The author was partially supported by Conacyt Grant No. 250845.}

\begin{abstract}
In this article we introduce a new type of local zeta functions and study some
connections with pseudodifferential operators in the framework of
non-Archimedean fields. The new local zeta functions are defined by
integra\-ting complex powers of norms of polynomials multiplied by infinitely
pseudo-differentiable functions. In characteristic zero, the new local zeta
functions admit meromorphic continuations to the whole complex plane, but they
are not rational functions. The real parts of the possible poles have a
description similar to the poles of Archimedean zeta functions. But they can
be irrational real numbers while in the classical case are rational numbers.
We also study, in arbitrary characteristic, certain connections between local
zeta functions and the existence of fundamental solutions for
pseudodifferential equations.

\end{abstract}
\maketitle

\section{Introduction}

This article aims to explore the connections between local zeta functions
(also called Igusa's local zeta functions) and pseudodifferential operators in
the framework of the non-Archimedean fields. The local zeta functions over
local fields, i.e. $\mathbb{R}$, $\mathbb{C}$, $\mathbb{Q}_{p}$,
$\mathbb{F}_{p}((T))$, are ubiquitous objects in mathematics and mathematical
physics, this is due mainly to the fact that they are the `dual objects' of
oscillatory integrals with analytic phases, see e.g. \cite{AVG},
\cite{Atiyah}, \cite{Ber}, \cite{B-G-Gonzalez-Dom}, \cite{Cassaigneetal},
\cite{D0}, \cite{D-L}, \cite{Gelfand-Shilov-1}, \cite{Igusa}, \cite{Igusa-SPF}%
, \cite{Igusa1}, \cite{Lo2}, \cite{Speer}, \cite{VZ}, \cite{VZ-1},
\cite{Weil}, \cite{Zuniga-LNM-2016} \cite{Zuniga-Nagoya-2003},
\cite{Zuniga-Padova-2003}, \cite{Zuniga-TAMS-2001} and the references therein.
Let $(K,\left\vert \cdot\right\vert _{K})$ be a local field of arbitrary
characteristic, $\phi:K^{n}\rightarrow\mathbb{C}$ a test function,
$\mathfrak{f}\in K\left[  x_{1},\ldots,x_{n}\right]  $ and $\left\vert
d^{n}x\right\vert _{K}$ a Haar measure on $K^{n}$. The simplest type of local
zeta function is defined as
\begin{equation}
Z_{\phi}\left(  s,\mathfrak{f}\right)  =\int\limits_{K^{n}\smallsetminus
\mathfrak{f}^{-1}\left(  0\right)  }\phi\left\vert \mathfrak{f}\right\vert
_{K}^{s}\left\vert d^{n}x\right\vert _{K}\text{ for }s\in\mathbb{C}\text{,
with }\operatorname{Re}(s)>0. \label{zeta_igusa}%
\end{equation}
These objects are deeply connected with string and Feynman amplitudes. Let us
mention that the works of Speer \cite{Speer} and Bollini, Giambiagi and
Gonz\'{a}lez Dom\'{\i}nguez \cite{B-G-Gonzalez-Dom} on regularization of
Feynman amplitudes in quantum field theory are based on the analytic
continuation of distributions attached to complex powers of polynomial
functions in the sense of Gel'fand and Shilov \cite{Gelfand-Shilov-1}. For
connections with string amplitudes see e.g. \cite{BG-GC-Zuniga}\ and the
references therein. In the Archimedean setting, the local zeta functions were
introduced in the 50's by Gel'fand and Shilov. The main motivation was that
the meromorphic continuation of Archimedean local zeta functions implies the
existence of fundamental solutions for differential ope\-rators with constant
coefficients. This result has a non-Archimedean counterpart. In
\cite{Zuniga-Padova-2003},\ see also \cite{Zuniga-LNM-2016} and the references
therein, the author noticed that the classical argument showing that the
analytic continuation of \ local zeta functions implies the existence of
fundamental solutions also works in non-Archimedean fields of characteristic
zero, and that for particular polynomials the Gel'fand-Shilov method gives
explicit formulas for fundamental solutions. In this article, we use methods
of pseudodifferential operators to study non-Archimedean local zeta functions.

A pseudodifferential operator with `polynomial symbol' $\left\vert
\mathfrak{h}\right\vert _{K}$, $\mathfrak{h}\in K\left[  \xi_{1},\ldots
,\xi_{n}\right]  $, is defined as $\boldsymbol{A}(\partial,\mathfrak{h}%
)\phi=\mathcal{F}_{\xi\rightarrow x}^{-1}(\left\vert \mathfrak{h}\right\vert
_{K}\mathcal{F}_{x\rightarrow\xi}\phi)$, where $\mathcal{F}$ denotes the
Fourier transform in the space of test functions. A theory of non-Archimedean
pseudodifferential equations is emerging motivated by its connections with
mathematical physics, see e.g. \cite{A-K-S}, \cite{Koch}, \cite{V-V-Z},
\cite{Zuniga-LNM-2016} and the references therein. The space of test functions
is not invariant under the action of pseudodifferential operators. We replace
it with $\mathcal{H}_{\infty}\subset L^{2}$ a Sobolev-type space which is a
nuclear countably Hilbert space in the sense of Gel'fand-Vilenkin. This type
of spaces was studied by the author in \cite{Zuniga2016}.

In this article we study the following integrals:%
\begin{equation}
Z_{\mathcal{F}\left(  g\right)  }\left(  s,\mathfrak{f}\right)  =\int
\limits_{K^{n}\smallsetminus\mathfrak{f}^{-1}\left(  0\right)  }\left\vert
\mathfrak{f}\right\vert _{K}^{s}\mathcal{F}\left(  g\right)  \left\vert
d^{n}x\right\vert _{K}\text{ } \label{zeta_zuiga}%
\end{equation}
for $s\in\mathbb{C}$, with $\operatorname{Re}(s)>0$, and $g\in\mathcal{H}%
_{\infty}$. For instance, if $\left\Vert \xi\right\Vert _{K}=\max
_{i}\left\vert \xi_{i}\right\vert _{K}$, $t>0$, and$\ \alpha>0$, then
$g\left(  x,t\right)  =\mathcal{F}_{\xi\rightarrow x}^{-1}\left(
e^{-t\left\Vert \xi\right\Vert _{K}^{\alpha}}\right)  \in\mathcal{H}_{\infty}%
$. This function is the `fundamental solution' of the heat equation over
$K^{n}$, see e.g. \cite{Koch}, \cite{V-V-Z}, \cite{Zuniga-LNM-2016}. We study
these local zeta functions in Section \ref{Sect_Local_zeta_forms} for certain
polynomials. It is interesting to mention that the complex counterparts of
these integrals are related with relevant arithmetic matters, see e.g.
\cite{Cassaigneetal} and the references therein. The space of test functions
is embedded in $\mathcal{H}_{\infty}$, and since the Fourier transform is an
isomorphism on this space, integrals of type (\ref{zeta_zuiga}) are
generalizations of the classical non-Archimedean local zeta functions
(\ref{zeta_igusa}). In characteristic zero, by using resolution of
singularities, we show that integrals $Z_{\mathcal{F}\left(  g\right)
}\left(  s,\mathfrak{f}\right)  $ admit meromorphic continuations to the whole
complex plane as $\mathcal{H}_{\infty}^{\ast}$-valued functions, here
$\mathcal{H}_{\infty}^{\ast}$\ denotes the strong dual of $\mathcal{H}%
_{\infty}$, see Theorem \ref{Theorem4}. These meromorphic continuations are
not rational functions of $q^{-s}$, see Section \ref{Sect_Local_zeta_forms},
and the description of the real parts of the possible poles resembles the case
of the Archimedean zeta functions, but there are relevant differences. If
$\left\{  Ni,v_{i}\right\}  _{i\in T}$ are the numerical data of an embedded
resolution of singularities of the map $\mathfrak{f}:K^{n}\rightarrow K$, then
the real parts of the possible poles of the meromorphic continuation of
$Z_{\mathcal{F}\left(  g\right)  }\left(  s,\mathfrak{f}\right)  $ belongs to
the set $\cup_{i\in T}\frac{-\left(  v_{i}+M_{i}\right)  }{N_{i}}$, where each
$M_{i}$ is an `arbitrary arithmetic progression of real numbers', see Theorem
\ref{Theorem4}. In the real case all the $M_{i}$ are just the set of
non-negative integers, see e.g. \cite{Igusa}, \cite{Igusa1}. The Hironaka
resolution of singularities theorem \cite{H} allow us to reduce the study of
$Z_{\mathcal{F}\left(  g\right)  }\left(  s,\mathfrak{f}\right)  $ to the case
in which $\mathfrak{f}$ is a monomial, like in the classical case. But the
study of these monomial integrals does \ not follow the classical pattern
because there is no a simply description for the functions in $\mathcal{H}%
_{\infty}$. All our results about the meromorphic continuation for monomial
integrals are valid in arbitrary characteristic, see Section
\ref{Sect_elelmetary_integrals}.

The pseudodifferential operators $\boldsymbol{A}(\partial,\mathfrak{h})$ give
rise to continuous operators from $\mathcal{H}_{\infty}$ \ onto itself, and
thus they have continuous adjoints, denoted as $\boldsymbol{A}^{\ast}%
(\partial,\mathfrak{h})$, from $\mathcal{H}_{\infty}^{\ast}$ \ onto itself. In
this framework, $Z_{\mathcal{F}\left(  \boldsymbol{A}(\partial,\mathfrak{h}%
_{i})g\right)  }\left(  s,\mathfrak{f}\right)  $ defines a $\mathcal{H}%
_{\infty}^{\ast}$-valued function for $s$ in the half-plane $\operatorname{Re}%
(s)>0$. Then, for instance, it makes sense to ask if $Z_{\mathcal{F}\left(
g\right)  }\left(  s,\mathfrak{f}\right)  $ satisfies a
pseudodifferential\ equation of the form
\[
\sum_{i=1}^{D}c_{i}\left(  q^{-s}\right)  Z_{\mathcal{F}\left(  \boldsymbol{A}%
(\partial,\mathfrak{h}_{i})g\right)  }\left(  s+k_{i},\mathfrak{f}\right)
=0,
\]
where the $c_{i}\left(  s\right)  \in\mathbb{C}\left(  q^{-s}\right)  $ and
the $k_{i}$ are integers. We also study the existence of fundamental
solutions, i.e. solutions for equations of the form $\boldsymbol{A}^{\ast
}(\partial,\mathfrak{f})E=\delta$ in $\mathcal{H}_{\infty}^{\ast}$, where
$\delta$ denotes the Dirac distribution. We show, like in the real case, see
for instance \cite{Atiyah}, \cite{Ber}, \cite{Igusa}, that the existence of a
fundamental solution is equivalent to the division problem: there exists $E$
in $\mathcal{H}_{\infty}^{\ast}$ such that $\widehat{E}\left\vert
\mathfrak{f}\right\vert _{K}=1$ almost everywhere, here $\widehat
{E}\mathfrak{\ }$denotes the Fourier of $E$ as a distribution, see Theorem
\ref{Theorem4A}. Finally, by using the Gel'fand-Shilov method of analytic
continuation, we show that the existence of an analytic continuation for
$Z_{\mathcal{F}\left(  g\right)  }\left(  s,\mathfrak{f}\right)  $ implies the
existence of a fundamental solution for operator $\boldsymbol{A}^{\ast
}(\partial,\mathfrak{f})$, see Theorem \ref{Theorem5}. These results are valid
in arbitrary characteristic.

Another important motivation for studying integrals $Z_{\mathcal{F}\left(
g\right)  }\left(  s,\mathfrak{f}\right)  $ comes from the fact that existence
of meromorphic continuations for local zeta functions in local fields of
positive characteristic is an open and difficult problem, see e.g.
\cite{Igusa-SPF}, \cite{Zuniga-Nagoya-2003}, \cite{Zuniga-TAMS-2001} and the
references therein. A natural and possible way to attack this problem is by
developing a suitable theory of $D$-modules in the framework of
non-Archimedean fields of arbitrary characteristic, which would allow\ us to
use Bernstein's approach to establish the meromorphic continuation for local
zeta functions in positive characteristic, see e.g. \cite{Ber}, \cite{Igusa}.
Several theories of arithmetic-type $D$-modules on fields of arbitrary
characteristic have been constructed, see e.g. \cite{Berthelot},
\cite{Mustata}. However, all these theories involved operators acting on
functions from $K^{n}$ into $K$, and the operators needed to study local zeta
functions must act on functions from $K^{n}$ into $\mathbb{C}$, thus, the only
possibility is to use pseudodifferential operators. Our results suggest the
existence of a theory of pseudodifferential $D$-modules \`{a} la Bernstein
which could be used to establish the analytic continuation of local zeta
functions in arbitrary characteristic.

\section{\label{Sect2}Fourier analysis on Non-Archimedean local fields:
essential ideas}

In this section we fix the notation and collect some basic definitions on
Fourier analysis on non-Archimedean local fields that we will use through the
article. For an in-depth exposition the reader may consult \cite{A-K-S},
\cite{Taibleson}, \cite{V-V-Z}, \cite{Weil}.

\subsection{Non-Archimedean local fields}

Along this article $K$ will denote a \textit{non-Archimedean local field} of
arbitrary characteristic unless otherwise stated. The associated absolute
value of $K$ is denoted as $\left\vert \cdot\right\vert _{K}$. \textit{The
ring of integers} of $K$ is $R_{K}=$ $\left\{  x\in K;\left\vert x\right\vert
_{K}\leq1\right\}  $, its unique maximal ideal is $P_{K}=$ $\left\{  x\in
R_{K};\left\vert x\right\vert _{K}<1\right\}  =\pi R_{K}$, where $\pi$ is a
fixed generator of $P_{K}$, typically called a \textit{local uniformizing
parameter} of $K$; and $R_{K}^{\times}=$ $\left\{  x\in R_{K};\left\vert
x\right\vert _{K}=1\right\}  $ is the \textit{group of units} of $R_{K}$. The
\textit{residue field} of $K$ is $R_{K}/P_{K}\simeq\mathbb{F}_{q}$, the finite
field with $q$ elements, where $q$ is a power of a prime number $p$. Let
$ord:K\rightarrow\mathbb{Z}\cup\left\{  \infty\right\}  $ denote the valuation
of $K$. We assume that for $x\in K^{\times}$, $\left\vert x\right\vert
_{K}=q^{-ord(x)}$, i.e. \ $\left\vert \cdot\right\vert _{K}$\ is a normalized
absolute value. Every non-Archimedean local field of characteristic zero is
isomorphic (as a topological field) to a finite extension of the field of
$p-$adic numbers $\mathbb{Q}_{p}$. And any non-Archimedean local field of
characteristic $p$ is isomorphic to a finite extension of the field of formal
Laurent series $\mathbb{F}_{q}((T))$ over a finite field $\mathbb{F}_{q}$, see
e.g. \cite{Weil}.

We extend the norm $||\cdot||_{K}$\ to $K^{n}$ by taking%
\[
||x||_{K}:=\max_{1\leq i\leq n}|x_{i}|_{K},\qquad\text{for }x=(x_{1}%
,\dots,x_{N})\in K^{n}.
\]
We define $ord(x)=\min_{1\leq i\leq N}\{ord(x_{i})\}$, then $||x||_{K}%
=p^{-ord(x)}$. The metric space $\left(  K^{n},||\cdot||_{K}\right)  $ is a
complete ultrametric space, which is a totally disconnected topological space.
For $l\in\mathbb{Z}$, denote by $B_{l}^{n}(a)=\{x\in K^{n};||x-a||_{K}\leq
q^{l}\}$ \textit{the ball of radius }$q^{l}$ \textit{with center at}
$a=(a_{1},\dots,a_{N})\in K^{n}$, and take $B_{l}^{n}(0):=B_{l}^{n}$. Note
that $B_{l}^{n}(a)=B_{l}(a_{1})\times\cdots\times B_{l}(a_{n})$, where
$B_{l}(a_{i}):=\{x\in K;|x-a_{i}|_{K}\leq q^{l}\}$ is the one-dimensional ball
of radius $q^{l}$ with center at $a_{i}\in K$. The ball $B_{0}^{n}$ equals the
product of $n$ copies of $B_{0}:=R_{K}$, the ring of integers of $K$. For
$l\in\mathbb{Z}$, denote by $S_{l}^{n}(a)=\{x\in K^{n};||x-a||_{K}=q^{l}\}$
\textit{the sphere of radius }$q^{l}$ \textit{with center at} $a=(a_{1}%
,\dots,a_{N})\in K^{n}$, and take $S_{l}^{N}(0):=S_{l}^{N}$.

\subsection{Some function spaces}

A complex-valued function $\phi$ defined on $K^{n}$ is \textit{called locally
constant} if for any $x\in K^{n}$ there exists an integer $l(x)\in\mathbb{Z}$
such that $\phi(x+x^{\prime})=\phi(x)$ for $x^{\prime}\in B_{l(x)}^{N}$. A
such function is called a \textit{Bruhat-Schwartz function (or a test
function)} if it has compact support. The $\mathbb{C}$-vector space of
Bruhat-Schwartz functions is denoted by $\mathcal{D}(K^{n}):=$ $\mathcal{D}$.
Let $\mathcal{D}^{^{\prime}}(K^{n}):=\mathcal{D}^{\prime}$ denote the set of
all continuous functionals (distributions) on $\mathcal{D}$.

Along this article, $\left\vert d^{n}x\right\vert _{K}$ will denote a Haar
measure on $K^{n}$ normalized so that $\int_{R_{K}^{n}}\left\vert
d^{n}x\right\vert _{K}=1$. Given $r\in\left[  0,\infty\right)  $, we denote by
$L^{r}\left(  K^{n},\left\vert d^{n}x\right\vert _{K}\right)  :=L^{r}%
(K^{n})=L^{r}$, the $\mathbb{C}$-vector space of all the complex valued
functions $g$ satisfying $\int_{K^{n}}\left\vert g\left(  x\right)
\right\vert ^{r}\left\vert d^{n}x\right\vert _{K}\allowbreak<\infty$;
$L^{\infty}\left(  K^{n},\left\vert d^{n}x\right\vert _{K}\right)
:=L^{\infty}(K^{n})=L^{\infty}$ denotes the $\mathbb{C}$-vector space of all
the complex valued functions $g$ such that the essential supremum of
$\left\vert g\right\vert $\ is bounded. \ Let denote by $C\left(
K^{n},\mathbb{C}\right)  :=C$, the $\mathbb{C}$-vector space of all the
complex valued functions which are continuous. Set
\[
C_{0}\left(  K^{n}\right)  :=\left\{  f:K^{n}\rightarrow\mathbb{C};f\text{ is
continuous and }\lim_{x\rightarrow\infty}f(x)=0\right\}  ,
\]
where $\lim_{x\rightarrow\infty}f(x)=0$ means that for every $\epsilon>0$
there exists a compact subset $B(\epsilon)$ such that $\left\vert f\left(
x\right)  \right\vert <\epsilon$ for $x\in K^{n}\smallsetminus B(\epsilon)$.
We recall that $\left(  C_{0}\left(  K^{n}\right)  ,\left\Vert \cdot
\right\Vert _{L^{\infty}}\right)  $ is a Banach space.

\subsection{Fourier transform}

We denote by $\chi(\cdot)$ a fixed additive character on $K$, i.e. a
continuous map from $K$ into the unit circle satisfying $\chi(y_{0}%
+y_{1})=\chi(y_{0})\chi(y_{1})$, $y_{0},y_{1}\in K$. If $x=\left(
x_{1},\ldots,x_{n}\right)  $, $\xi=\left(  \xi_{1},\ldots,\xi_{n}\right)  $,
we set $x\cdot\xi=\sum_{i=1}^{n}x_{i}\xi_{i}$.

If $g\in L^{1}$ its Fourier transform is defined by%
\[
(\mathcal{F}g)(\xi)=%
{\displaystyle\int\limits_{K^{n}}}
g\left(  x\right)  \chi\left(  -x\cdot\xi\right)  \left\vert d^{n}x\right\vert
_{K}=%
{\displaystyle\int\limits_{K^{n}}}
g\left(  x\right)  \overline{\chi\left(  x\cdot\xi\right)  }\left\vert
d^{n}x\right\vert _{K}\text{,}%
\]
where the bar denotes the complex conjugate. The Fourier transform is an
isomorphism of $\mathbb{C}$-vector spaces from $\mathcal{D}\left(
K^{n}\right)  $ into itself satisfying
\begin{equation}
(\mathcal{F}(\mathcal{F}g))(x)=g(-x) \label{inv_Fourier_transform}%
\end{equation}
for every $g$ in $\mathcal{D}\left(  K^{n}\right)  $. If $g\in L^{2}$, its
Fourier transform is defined as%
\[
(\mathcal{F}g)(\xi)=\lim_{l\rightarrow\infty}%
{\displaystyle\int\limits_{\left\Vert x\right\Vert _{K}\leq q^{l}}}
g\left(  x\right)  \chi\left(  -x\cdot\xi\right)  \left\vert d^{n}x\right\vert
_{K},
\]
where the \ limit is taken in $L^{2}$. We recall that the Fourier transform is
unitary on $L^{2}$, i.e. $\left\Vert g\right\Vert _{L^{2}}=\left\Vert
\mathcal{F}g\right\Vert _{L^{2}}$ for $g\in L^{2}$ and that
(\ref{inv_Fourier_transform}) is also valid in $L^{2}$, see e.g. \cite[Chapter
III, Section 2]{Taibleson}. We will also use the notation $\mathcal{F}%
_{x\rightarrow\xi}g$ and $\widehat{g}$\ for the Fourier transform of $g$.

The Fourier transform $\mathcal{F}\left(  T\right)  $ of a distribution
$T\in\mathcal{D}^{^{\prime}}\left(  K^{n}\right)  $ is defined by%
\[
\left(  \mathcal{F}\left(  T\right)  ,g\right)  =\left(  T,\mathcal{F}\left(
g\right)  \right)  \text{ for all }g\in\mathcal{D}^{\prime}\left(
K^{n}\right)  \text{.}%
\]
The Fourier transform $T\rightarrow\mathcal{F}\left(  T\right)  $ is a linear
isomorphism from $\mathcal{D}^{^{\prime}}\left(  K^{n}\right)  $\ onto itself.
Furthermore, $T=\mathcal{F}\left[  \mathcal{F}\left[  T\right]  \left(
-\xi\right)  \right]  $. We also use the notation $\mathcal{F}_{x\rightarrow
\xi}T$ and $\widehat{T}$\ for the Fourier transform of $T$.

\section{\label{Sect3}The spaces $\mathcal{H}_{\infty}$}

The Bruhat-Schwartz space $\mathcal{D}(K^{n}\mathbb{)}$ is not invariant under
the action of pseudo\-differential operators. In this section, we review and
expand some results about a class of nuclear countably Hilbert spaces
introduced by the author in \cite{Zuniga2016}, these spaces are invariant
under the action of large class of pseudodifferential operators. The notation
here is slightly different to the notation used in \cite{Zuniga2016}, in
addition, the results in \cite{Zuniga2016} were formulated for $\mathbb{Q}%
_{p}^{n}$, but these results are valid in non-Archimedean local fields of
arbitrary characteristic. For an in-depth discussion about nuclear countably
Hilbert spaces, the reader may consult \cite{Gel-Shilov}, \cite{Gel-Vil},
\cite{Hida et al}, \cite{Obata}.

\begin{notation}
We set $\mathbb{R}_{+}:=\left\{  x\in\mathbb{R}:x\geq0\right\}  $. We denote
by $\mathbb{N}$ the set of non-negative integers. We set $\left[  \xi\right]
_{K}:=\max\left(  1,\left\Vert \xi\right\Vert _{K}\right)  $.
\end{notation}

We define for $\varphi$, $\varrho$ in $\mathcal{D}(K^{n})$ the following
scalar product:%
\begin{equation}
\left\langle \varphi,\varrho\right\rangle _{l}:=%
{\textstyle\int\limits_{K^{n}}}
\left[  \xi\right]  _{K}^{l}\widehat{\varphi}\left(  \xi\right)
\overline{\widehat{\varrho}}\left(  \xi\right)  \left\vert d^{n}\xi\right\vert
_{K}, \label{product}%
\end{equation}
for $l\in\mathbb{N}$, where the bar denotes the complex conjugate. We also set
$\left\Vert \varphi\right\Vert _{l}^{2}=\left\langle \varphi,\varphi
\right\rangle _{l}$. Notice that $\left\Vert \cdot\right\Vert _{l}%
\leq\left\Vert \cdot\right\Vert _{m}$ for $l\leq m$. Let denote by
$\mathcal{H}_{l}\left(  K^{n}\right)  :=\mathcal{H}_{l}$ the completion of
$\mathcal{D}(K^{n})$ with respect to $\left\langle \cdot,\cdot\right\rangle
_{l}$. Then $\mathcal{H}_{m}\hookrightarrow\mathcal{H}_{l}$ (continuous
embedding) for $l\leq m$. We set%
\[
\mathcal{H}_{\infty}\left(  K^{n}\right)  :=\mathcal{H}_{\infty}=%
{\textstyle\bigcap\limits_{l\in\mathbb{N}}}
\mathcal{H}_{l}.
\]
Notice that $\mathcal{H}_{0}=L^{2}$ and that $\mathcal{H}_{\infty}\subset
L^{2}$. With the topology induced by the family of seminorms $\left\Vert
\cdot\right\Vert _{l\in\mathbb{N}}$, $\mathcal{H}_{\infty}$ becomes a locally
convex space, which is metrizable. Indeed,
\[
d\left(  f,g\right)  :=\max_{l\in\mathbb{N}}\left\{  2^{-l}\frac{\left\Vert
f-g\right\Vert _{l}}{1+\left\Vert f-g\right\Vert _{l}}\right\}  \text{, with
}f\text{, }g\in\mathcal{H}_{\infty}\text{,}%
\]
is a metric for the topology of $\mathcal{H}_{\infty}$ considered as a convex
topological space. A sequence $\left\{  f_{l}\right\}  _{l\in\mathbb{N}}$ in
$\left(  \mathcal{H}_{\infty},d\right)  $ converges to $f\in\mathcal{H}%
_{\infty}$, if and only if, $\left\{  f_{l}\right\}  _{l\in\mathbb{N}}$
converges to $f$ in the norm $\left\Vert \cdot\right\Vert _{l}$ for all
$l\in\mathbb{N}$. From this observation follows that the topology on
$\mathcal{H}_{\infty}$ coincides with the projective\ limit topology $\tau
_{P}$. An open neighborhood base at zero of $\tau_{P}$ is given by the choice
of $\epsilon>0$ and $l\in\mathbb{N}$, and the set
\[
U_{\epsilon,l}:=\left\{  f\in\mathcal{H}_{\infty};\left\Vert f\right\Vert
_{l}<\epsilon\right\}  .
\]
The space $\mathcal{H}_{\infty}$ endowed with the topology $\tau_{P}$\ is a
countably Hilbert space in the sense of Gel'fand and Vilenkin, see e.g.
\cite[Chapter I, Section 3.1]{Gel-Vil} or \cite[Section 1.2]{Obata}.
Furthermore $\left(  \mathcal{H}_{\infty},\tau_{P}\right)  $ is metrizable and
complete and hence a Fr\'{e}chet space, cf. Lemma \cite[Lemma 3.3]%
{Zuniga2016}. In addition, the completion of the metric space $\left(
\mathcal{D}(K^{n}),d\right)  $ is $\left(  \mathcal{H}_{\infty},d\right)  $,
and this space is a nuclear countably Hilbert space, cf. \cite[Lemma 3.4,
Theorem 3.6]{Zuniga2016}.

\begin{lemma}
\label{Lemma2A}With the above notation, the following assertions hold:

\noindent(i) $\mathcal{H}_{\infty}\left(  K^{n}\right)  $ is continuously
embedded in $C_{0}\left(  K^{n}\right)  $;

\noindent(ii) $\mathcal{H}_{l}\left(  K^{n}\right)  =\left\{  f\in L^{2}%
(K^{n});\left\Vert f\right\Vert _{l}<\infty\right\}  =\left\{  T\in
\mathcal{D}^{\prime}(K^{n});\left\Vert T\right\Vert _{l}<\infty\right\}  $;

\noindent(iii) $\mathcal{H}_{\infty}\left(  K^{n}\right)  =\left\{  f\in
L^{2}(K^{n});\left\Vert f\right\Vert _{l}<\infty\text{, for every }%
l\in\mathbb{N}\right\}  $;

\noindent(iv) $\mathcal{H}_{\infty}\left(  K^{n}\right)  =\left\{
T\in\mathcal{D}^{\prime}(K^{n});\left\Vert T\right\Vert _{l}<\infty\text{, for
every }l\in\mathbb{N}\right\}  $. The equalities in (ii)-(iv) are in the sense
of vector spaces.

\noindent(v) $\mathcal{H}_{\infty}\left(  K^{n}\right)  \subset L^{1}\left(
K^{n}\right)  $. In particular, $\widehat{g}\in C_{0}\left(  K^{n}\right)  $
for $g\in\mathcal{H}_{\infty}\left(  K^{n}\right)  $.
\end{lemma}

\begin{proof}
(i) Take $f\in\mathcal{H}_{\infty}$ and $l>n$, then by using Cauchy-Schwarz
inequality,
\begin{align*}
\left\Vert \widehat{f}\right\Vert _{L^{1}}  &  =\int\limits_{K^{n}}\left\vert
\widehat{f}\right\vert \left\vert d^{n}\xi\right\vert _{K}=\int\limits_{K^{n}%
}\left\{  \left[  \xi\right]  _{K}^{\frac{l}{2}}\left\vert \widehat
{f}\right\vert \right\}  \left\{  \frac{1}{\left[  \xi\right]  _{K}^{\frac
{l}{2}}}\right\}  \left\vert d^{n}\xi\right\vert _{K}\\
&  \leq\left\Vert \frac{1}{\left[  \xi\right]  _{K}^{\frac{l}{2}}}\right\Vert
_{L^{2}}\left\Vert \left[  \xi\right]  _{K}^{\frac{l}{2}}\left\vert
\widehat{f}\right\vert \right\Vert _{L^{2}}\leq C(n,l)\left\Vert f\right\Vert
_{l},
\end{align*}
where $C(n,l)$\ is a positive constant, which shows that $\widehat{f}\in
L^{1}$. Then, $f$ is continuous and by the Riemann-Lebesgue theorem, (see e.g.
\cite[Theorem 1.6]{Taibleson}), $f\in C_{0}\left(  K^{n}\right)  $. On the
other hand, $\left\Vert f\right\Vert _{L^{\infty}}\leq\left\Vert \widehat
{f}\right\Vert _{L^{1}}\leq C(n,l)\left\Vert f\right\Vert _{l}$, which shows
that $\mathcal{H}_{l}$ is conti\-nuously embedded in $C_{0}\left(
K^{n}\right)  $ for $l>n$. Thus $\mathcal{H}_{\infty}\subset C_{0}\left(
K^{n}\right)  $. Now, if $f_{m}$ $\underrightarrow{d}$ $f$ in $\mathcal{H}%
_{\infty}\left(  K^{n}\right)  $, i.e. if $f_{m}$ $\underrightarrow{\left\Vert
\cdot\right\Vert _{l}}$ $f$ in $\mathcal{H}_{l}$ for any $l\in\mathbb{N}$,
then $f_{m}$ $\underrightarrow{\left\Vert \cdot\right\Vert _{L^{\infty}}}$ $f$
in $C_{0}\left(  K^{n}\right)  $.

(ii) In order to prove the first equality, it is sufficient to show that if
$f\in L^{2}$ and $\left\Vert f\right\Vert _{l}<\infty$ then $f\in
\mathcal{H}_{l}$. The condition $\left\Vert f\right\Vert _{l}<\infty$ is
equivalent to $\left[  \xi\right]  _{K}^{\frac{l}{2}}\widehat{f}\in
L^{2}(K^{n})$, which implies that $\widehat{\left\{  \left[  \xi\right]
_{K}^{\frac{l}{2}}\widehat{f}\right\}  }\left(  -\xi\right)  \in L^{2}(K^{n}%
)$. By the density of $\mathcal{D}(K^{n})$ in $L^{2}(K^{n})$, there is a
sequence $\left\{  g_{k}\right\}  _{k\in\mathbb{N}}$ in $\mathcal{D}(K^{n})$
such that $g_{k}\left(  \xi\right)  $ $\underrightarrow{\left\Vert
\cdot\right\Vert _{L^{2}}}$ $\widehat{\left\{  \left[  \xi\right]  _{K}%
^{\frac{l}{2}}\widehat{f}\right\}  }\left(  -\xi\right)  $, which implies that
$\widehat{g}_{k}$ $\underrightarrow{\left\Vert \cdot\right\Vert _{L^{2}}}$
$\left[  \xi\right]  _{K}^{\frac{l}{2}}\widehat{f}$, which is equivalent to
$\mathcal{F}^{-1}(\widehat{g}_{k}/\left[  \xi\right]  _{K}^{\frac{l}{2}%
})\underrightarrow{\left\Vert \cdot\right\Vert _{l}}$ $f$ with $\widehat
{g}_{k}/\left[  \xi\right]  _{K}^{\frac{l}{2}}\in\mathcal{D}(K^{n})$ for any
$k\in\mathbb{N}$. To establish the second equality, we note that since
$\left\Vert \cdot\right\Vert _{0}\leq\left\Vert \cdot\right\Vert _{l}$ for any
$l\in\mathbb{N}$, if $T\in H_{l}$ then $\widehat{T}\in L^{2}$, and thus
$T\in\mathcal{D}^{\prime}(K^{n})$ and $\left\Vert T\right\Vert _{l}<\infty$.
Conversely, if $T\in\mathcal{D}^{\prime}(K^{n})$ and $\left\Vert T\right\Vert
_{l}<\infty$ then $T\in L^{2}$ and $\left\Vert T\right\Vert _{l}<\infty$.

(iii) It follows from (ii).

(iv) It follows from (iii) by using that the following assertions are
equivalent: (1) $T\in\mathcal{D}^{\prime}(K^{n})$ and $\left\Vert T\right\Vert
_{l}<\infty$ for any $l\in\mathbb{N}$; (2) $T\in L^{2}$ and $\left\Vert
T\right\Vert _{l}<\infty$ for any $l\in\mathbb{N}$.

(iv) By Theorem 3.15-(ii) in \cite{Zuniga2016}, $\mathcal{H}_{\infty}\left(
K^{n}\right)  \subset L^{1}\left(  K^{n}\right)  $. The fact that $\widehat
{g}\in C_{0}\left(  K^{n}\right)  $ for $g\in\mathcal{H}_{\infty}\left(
K^{n}\right)  $ follows from the Riemann-Lebesgue theorem.
\end{proof}

\subsection{The dual space of $\mathcal{H}_{\infty}$}

For $m\in\mathbb{N}$ and $T\in\mathcal{D}^{\prime}(K^{n})$, we set
\[
\left\Vert T\right\Vert _{-m}^{2}:=%
{\textstyle\int\limits_{K^{n}}}
\left[  \xi\right]  _{K}^{-m}\left\vert \widehat{T}\left(  \xi\right)
\right\vert ^{2}\left\vert d^{n}\xi\right\vert _{K}.
\]
Then $\mathcal{H}_{-m}:=\mathcal{H}_{-m}\left(  K^{n}\right)  =\left\{
T\in\mathcal{D}^{\prime}(K^{n});\left\Vert T\right\Vert _{-m}^{2}%
<\infty\right\}  $ is a complex Hilbert space. Denote by $\mathcal{H}%
_{m}^{\ast}$ the strong dual space of $\mathcal{H}_{m}$. It is useful to
suppress the correspondence between $\mathcal{H}_{m}^{\ast}$ and
$\mathcal{H}_{m}$ given by the Riesz theorem. Instead we identify
$\mathcal{H}_{m}^{\ast}$ and $\mathcal{H}_{-m}$ by associating $T\in
\mathcal{H}_{-m}$ with the functional on $\mathcal{H}_{m}$ given by
\begin{equation}
\left[  T,g\right]  :=\int\limits_{K^{n}}\overline{\widehat{T}\left(
\xi\right)  }\widehat{g}\left(  \xi\right)  \left\vert d^{n}\xi\right\vert
_{K}. \label{pairing}%
\end{equation}
Notice that
\begin{equation}
\left\vert \left[  T,g\right]  \right\vert \leq\left\Vert T\right\Vert
_{-m}\left\Vert g\right\Vert _{m}. \label{bound_T}%
\end{equation}
Now by a well-known result in the theory of countable Hilbert spaces, see e.g.
\cite[Chapter I, Section 3.1]{Gel-Vil}, $\mathcal{H}_{0}^{\ast}\subset
\mathcal{H}_{1}^{\ast}\subset\ldots\subset\mathcal{H}_{m}^{\ast}\subset\ldots$
and%
\begin{equation}
\mathcal{H}_{\infty}^{\ast}\left(  K^{n}\right)  =\mathcal{H}_{\infty}^{\ast
}=\bigcup\limits_{m\in\mathbb{N}}\mathcal{H}_{-m}=\left\{  T\in\mathcal{D}%
^{\prime}(K^{n});\left\Vert T\right\Vert _{-l}<\infty\text{, for some }%
l\in\mathbb{N}\right\}  \label{H_infinity_*}%
\end{equation}
as vector spaces. We mention that since $\mathcal{H}_{\infty}$ is a nuclear
space, cf. \cite[Lemma 3.4, Theorem 3.6]{Zuniga2016}, the weak and strong
convergence are equivalent in $\mathcal{H}_{\infty}^{\ast}$, see e.g.
\cite[Chapter I, Section 6, Theorem 6.4]{Gel-Shilov}. We consider
$\mathcal{H}_{\infty}^{\ast}$ endowed with the strong topology. On the other
hand, let $B:\mathcal{H}_{\infty}^{\ast}\times\mathcal{H}_{\infty}%
\rightarrow\mathbb{C}$ be a bilinear functional. \ Then $B$ is continuous in
each of its arguments if and only if there exist norms $\left\Vert
\cdot\right\Vert _{m}^{(a)}$ in $\mathcal{H}_{m}^{\ast}$ and $\left\Vert
\cdot\right\Vert _{l}^{(b)}$ in $\mathcal{H}_{l}$ such that $\left\vert
B\left(  T,g\right)  \right\vert \leq M\left\Vert T\right\Vert _{m}%
^{(a)}\left\Vert g\right\Vert _{l}^{(b)}$ with $M$ a positive constant
independent of $T$ and $g$, see e.g. \cite[Chapter I, Section 1.2]{Gel-Vil}
and \cite[Chapter I, Section 4.1]{Gel-Shilov}. This implies that
(\ref{pairing}) is a continuous bilinear form on $\mathcal{H}_{\infty}^{\ast
}\times\mathcal{H}_{\infty}$, which we will use as a paring between
$\mathcal{H}_{\infty}^{\ast}$ and $\mathcal{H}_{\infty}$.

\begin{remark}
The spaces $\mathcal{H}_{\infty}\subset L^{2}\subset\mathcal{H}_{\infty}%
^{\ast}$\ form a Gel'fand triple (also called a rigged Hilbert space), i.e.
$\mathcal{H}_{\infty}$ is a nuclear space which is densely and continuously
embedded in $L^{2}$ and $\left\Vert g\right\Vert _{L^{2}}^{2}=\left[
g,g\right]  $. This Gel'fand triple was introduced in \cite{Zuniga2016}.
\end{remark}

\begin{remark}
\label{Note_Dirac_dist}By the proof of Lemma \ref{Lemma2A}-(i), if
$g\in\mathcal{H}_{\infty}$, then $\widehat{g}\in L^{1}\cap L^{2}$ and by the
dominated convergence theorem, $g\left(  0\right)  =\int\widehat{g}\left\vert
d^{n}\xi\right\vert _{K}$. Consequently,
\[
\left[  \widehat{1},g\right]  =\int\limits_{K^{n}}\widehat{g}\left(
\xi\right)  \left\vert d^{n}\xi\right\vert _{K}=g\left(  0\right)  ,
\]
and thus $\widehat{1}$\ defines an element of $\mathcal{H}_{\infty}^{\ast}$,
which we identify with the Dirac distribution $\delta$, i.e. $\left[
\delta,g\right]  =g\left(  0\right)  $. In addition, $\delta\ast g=g$ for any
$g\in\mathcal{H}_{\infty}$. Indeed, take $g_{n}$ $\underrightarrow{\left\Vert
\cdot\right\Vert _{l}}$ $g$ for any $l\in\mathbb{N}$, with $\left\{
g_{n}\right\}  _{n\in\mathbb{N}}$ in $\mathcal{D}\left(  K^{n}\right)  $ and
$g\in\mathcal{H}_{\infty}$. Then $\left\Vert \delta\ast g_{n}-g\right\Vert
_{l}=\left\Vert g_{n}-g\right\Vert _{l}\rightarrow0$, since $\delta\ast
g_{n}=g_{n}$, for any $l\in\mathbb{N}$, which means that $g\rightarrow$
$\delta\ast g$ is continuous in $\mathcal{D}\left(  K^{n}\right)  $, which is
dense in $\mathcal{H}_{\infty}$.
\end{remark}

\subsection{\label{Sect.Oper_H_infinity}Pseudodifferential operators acting on
$\mathcal{H}_{\infty}$}

Let $\mathfrak{h}_{i}$ be a non-constant polynomial in $R_{K}\left[  \xi
_{1},\ldots,\xi_{n}\right]  $ of degree $d_{i}$, \ for $i=1,\ldots,r$, with
$1\leq r\leq n$, and let $\alpha_{i}$ be a complex number such that
$\operatorname{Re}(\alpha_{i})>0$ for $i=1,\ldots,r$. We set $\underline
{\mathfrak{h}}=\left(  \mathfrak{h}_{1},\ldots,\mathfrak{h}_{r}\right)  $ and
$\boldsymbol{\alpha}=\left(  \alpha_{1},\ldots,\alpha_{r}\right)  $ and attach
them the following pseudodifferential operator:%
\[%
\begin{array}
[c]{ccc}%
\mathcal{D}(K^{n}) & \rightarrow & L^{2}\cap C\\
\varphi & \rightarrow & \boldsymbol{P}\left(  \partial,\underline
{\mathfrak{h}},\boldsymbol{\alpha}\right)  \varphi,
\end{array}
\]
where $\left(  \boldsymbol{P}\left(  \partial,\underline{\mathfrak{h}%
},\boldsymbol{\alpha}\right)  \varphi\right)  \left(  x\right)  =\mathcal{F}%
_{\xi\rightarrow x}^{-1}\left(  \prod\nolimits_{i=1}^{r}\left\vert
\mathfrak{h}_{i}\left(  \xi\right)  \right\vert _{K}^{\alpha_{i}}%
\mathcal{F}_{x\rightarrow\xi}\varphi\right)  $.

\begin{notation}
For $t\in\mathbb{R}$, we denote by $\left\lceil t\right\rceil :=\min\left\{
m\in\mathbb{Z};m\geq t\right\}  $, the ceiling function.
\end{notation}

\begin{lemma}
\label{lemma3A} The mapping $\boldsymbol{P}\left(  \partial,\underline
{\mathfrak{h}},\boldsymbol{\alpha}\right)  :\mathcal{H}_{\infty}%
\rightarrow\mathcal{H}_{\infty}$ is a well-defined continuous operator between
locally convex spaces.
\end{lemma}

\begin{proof}
The result follows from the following assertion:

\textbf{Claim.} $\boldsymbol{P}\left(  \partial,\underline{\mathfrak{h}%
},\boldsymbol{\alpha}\right)  :\mathcal{H}_{m(l)}\rightarrow\mathcal{H}_{l}$,
with $m(l):=l+2\sum_{i=1}^{r}d_{i}\left\lceil \operatorname{Re}(\alpha
_{i})\right\rceil $, defines a continuous operator between Banach spaces.

Indeed, by the Claim, if $g\in\mathcal{H}_{\infty}$, then $\boldsymbol{P}%
\left(  \partial,\underline{\mathfrak{h}},\boldsymbol{\alpha}\right)
g\in\mathcal{H}_{\infty}$. To check the continuity, we take a sequence
$\left\{  g_{k}\right\}  _{k\in\mathbb{N}}$ in $\mathcal{H}_{\infty}$ such
that $g_{k}$ $\underrightarrow{d}$ $g$, with $g\in\mathcal{H}_{\infty}$, i.e.
$g_{k}$ $\underrightarrow{\left\Vert \cdot\right\Vert _{l}}$ $g$ for any
$l\in\mathbb{N}$. By the Claim%
\[
\left\Vert \boldsymbol{P}\left(  \partial,\underline{\mathfrak{h}%
},\boldsymbol{\alpha}\right)  g_{k}-\boldsymbol{P}\left(  \partial
,\underline{\mathfrak{h}},\boldsymbol{\alpha}\right)  g\right\Vert _{l}%
\leq\left\Vert g_{k}-g\right\Vert _{m(l)}.
\]

which implies that $\boldsymbol{P}\left(  \partial,\underline{\mathfrak{h}%
},\boldsymbol{\alpha}\right)  g_{k}$ $\underrightarrow{\left\Vert
\cdot\right\Vert _{l}}$ $\boldsymbol{P}\left(  \partial,\underline
{\mathfrak{h}},\boldsymbol{\alpha}\right)  g$ for any $l\in\mathbb{N}$.

\textbf{Proof of the Claim. }By taking $\varphi\in\mathcal{D}(K^{n})$,\ and
\ using that
\begin{equation}
\left\vert \mathfrak{h}_{i}\left(  \xi\right)  \right\vert _{K}%
^{\operatorname{Re}(\alpha_{i})}\leq\left[  \xi\right]  _{K}^{d_{i}\left\lceil
\operatorname{Re}(\alpha_{i})\right\rceil }, \label{Estimation_h}%
\end{equation}
where $d_{i}$ denotes the degree of $\mathfrak{h}_{i}$, we have%
\begin{gather*}
\left\Vert \boldsymbol{P}\left(  \partial,\underline{\mathfrak{h}%
},\boldsymbol{\alpha}\right)  g\right\Vert _{l}^{2}=%
{\textstyle\int\limits_{K^{n}}}
\left[  \xi\right]  ^{l}\prod\nolimits_{i=1}^{r}\left\vert \mathfrak{h}%
_{i}\left(  \xi\right)  \right\vert _{K}^{2\operatorname{Re}(\alpha_{i}%
)}\left\vert \widehat{\varphi}\left(  \xi\right)  \right\vert ^{2}\left\vert
d^{n}\xi\right\vert _{K}\\
\leq%
{\textstyle\int\limits_{K^{n}}}
\left[  \xi\right]  ^{l+2\sum_{i=1}^{r}d_{i}\left\lceil \operatorname{Re}%
(\alpha_{i})\right\rceil }\left\vert \widehat{\varphi}\left(  \xi\right)
\right\vert ^{2}\left\vert d^{n}\xi\right\vert _{K}\leq\left\Vert
\varphi\right\Vert _{l+2\sum_{i=1}^{r}d_{i}\left\lceil \operatorname{Re}%
(\alpha_{i})\right\rceil }^{2}.
\end{gather*}
Now, from the density of $\mathcal{D}(K^{n})$ in $\mathcal{H}_{l+2\sum
_{i=1}^{r}d_{i}\left\lceil \operatorname{Re}(\alpha_{i})\right\rceil }$, we
conclude that
\[
\boldsymbol{P}\left(  \partial,\underline{\mathfrak{h}},\boldsymbol{\alpha
}\right)  :\mathcal{H}_{l+2\sum_{i=1}^{r}d_{i}\left\lceil \operatorname{Re}%
(\alpha_{i})\right\rceil }\rightarrow\mathcal{H}_{l}%
\]
defines a continuous operator between Banach spaces.
\end{proof}

\subsubsection{\label{Sec_dual_operator}Adjoint operators on $\mathcal{H}%
_{\infty}$}

By using that $\boldsymbol{P}\left(  \partial,\underline{\mathfrak{h}%
},\boldsymbol{\alpha}\right)  :\mathcal{H}_{\infty}\rightarrow\mathcal{H}%
_{\infty}$ is a continuous operator and some results on adjoint operators in
the setting of locally convex spaces, see e.g. \ \cite[Chapter VII, Section
1]{Yosida}, one gets that there exists a continuous operator $\boldsymbol{P}%
^{\ast}\left(  \partial,\underline{\mathfrak{h}},\boldsymbol{\alpha}\right)
:\mathcal{H}_{\infty}^{\ast}\rightarrow\mathcal{H}_{\infty}^{\ast}$
satisfying
\[
\left[  \boldsymbol{P}^{\ast}\left(  \partial,\underline{\mathfrak{h}%
},\boldsymbol{\alpha}\right)  T,g\right]  =\left[  T,\boldsymbol{P}\left(
\partial,\underline{\mathfrak{h}},\boldsymbol{\alpha}\right)  g\right]  \text{
}%
\]
for any $T\in\mathcal{H}_{\infty}^{\ast}$ and any $g\in\mathcal{H}_{\infty}$.
We call $\boldsymbol{P}^{\ast}$ the \textit{adjoint operator} of
$\boldsymbol{P}$.

\subsection{Some additional results}

\begin{lemma}
\label{lemma4B}Take $\underline{\mathfrak{h}}=\left(  \mathfrak{h}_{1}%
,\ldots,\mathfrak{h}_{r}\right)  $ with $\mathfrak{h}_{i}\in R_{K}\left[
\xi_{1},\ldots,\xi_{n}\right]  \smallsetminus R_{K}$, $1\leq r\leq n$, and
$\boldsymbol{\alpha=}\left(  \alpha_{1},\ldots,\alpha_{r}\right)
\in\mathbb{C}^{r}$ with $\operatorname{Re}(\alpha_{i})>0$ for any $i$, as
before, and $g\in\mathcal{H}_{\infty}$, and define%
\[
I_{g}\left(  \boldsymbol{\alpha},\underline{\mathfrak{h}}\right)
=\int\limits_{K^{n}\smallsetminus\cup_{i=1}^{r}\mathfrak{h}_{i}^{-1}\left(
0\right)  }\prod\limits_{i=1}^{r}\left\vert \mathfrak{h}_{i}\left(
\xi\right)  \right\vert _{K}^{\alpha_{i}}\widehat{g}\left(  \xi\right)
\left\vert d^{n}\xi\right\vert _{K}.
\]
Then $I_{g}\left(  \boldsymbol{\alpha},\underline{\mathfrak{h}}\right)  $
defines a $\mathcal{H}_{\infty}^{\ast}$-valued holomorphic function of
$\boldsymbol{\alpha}$ in the half-plane $\operatorname{Re}(\alpha)>0$ for
$i=1,\ldots,r$.
\end{lemma}

\begin{proof}
By using (\ref{Estimation_h}), we get%
\[
\left\vert I_{g}\left(  \boldsymbol{\alpha},\underline{\mathfrak{h}}\right)
\right\vert \leq\int\limits_{K^{n}}\left[  \xi\right]  _{K}^{\sum_{i=1}%
^{r}d_{i}\left\lceil \operatorname{Re}\left(  \alpha_{i}\right)  \right\rceil
}\left\vert \widehat{g}\left(  \xi\right)  \right\vert \left\vert d^{n}%
\xi\right\vert _{K},
\]
and by proof of Lemma \ref{Lemma2A}-(i),%
\[
\left\vert I_{g}\left(  \boldsymbol{\alpha},\underline{\mathfrak{h}}\right)
\right\vert \leq C(n,l)\left\Vert g\right\Vert _{l+2\sum_{i=1}^{r}%
d_{i}\left\lceil \operatorname{Re}(\alpha_{i})\right\rceil }^{2}%
\]
for any positive integer $l>n$. This implies that $I_{g}\left(
\boldsymbol{\alpha},\underline{\mathfrak{h}}\right)  $ is a $\mathcal{H}%
_{\infty}^{\ast}$-valued function if $\boldsymbol{\alpha}$ \ belongs to the
half-plane $\operatorname{Re}(\alpha_{i})>0$, $i=1,\ldots,r$. To establish the
holomorphy of $I_{g}\left(  \boldsymbol{\alpha},\underline{\mathfrak{h}%
}\right)  $, we recall that a continuous complex-valued function defined in an
open set $A\subseteq$ $\mathbb{C}^{r}$, which is holomorphic in each variable
separately, is holomorphic in $A$. Thus, it is sufficient to show that
$I_{g}\left(  \boldsymbol{\alpha},\underline{\mathfrak{h}}\right)  $ is
holomorphic in each $\alpha_{i}$. This last fact follows from a classical
argument, see e.g. \cite[Lemma 5.3.1]{Igusa}, by showing the existence of a
function $\Phi_{\mathcal{K}}\left(  \xi\right)  \in L^{1}$ such that for any
compact subset $\mathcal{K}$ of $\left\{  \alpha_{i}\in\mathbb{C}%
;\operatorname{Re}(\alpha_{i})>0\right\}  $, with $\alpha_{j}$ fixed for
$j\neq i$, it verifies that $\prod\nolimits_{i=1}^{r}\left\vert \mathfrak{h}%
_{i}\left(  \xi\right)  \right\vert _{K}^{\operatorname{Re}(\alpha_{i}%
)}\left\vert \widehat{g}\left(  \xi\right)  \right\vert \leq\Phi_{\mathcal{K}%
}\left(  \xi\right)  $ for $\alpha_{i}\in\mathcal{K}$.
\end{proof}

\begin{notation}
If $I$ is a finite set, then $\left\vert I\right\vert $ denotes its cardinality.
\end{notation}

\begin{lemma}
\label{lemma4B_1}Let $I$ and $J$ be two non-empty subsets of $\left\{
1,\ldots,n\right\}  $ such that $I\cap J=\emptyset$ and $I\cup J=\left\{
1,\ldots,n\right\}  $. Set $x=\left(  x_{1},\ldots,x_{n}\right)  =\left(
x_{I},x_{J}\right)  \in K^{\left\vert I\right\vert }\times K^{\left\vert
J\right\vert }$ with $x_{I}=\left(  x_{i}\right)  _{i\in I}$ and
$x_{J}=\left(  x_{i}\right)  _{i\in J}$. With this notation, the measure
$\left\vert d^{n}x\right\vert _{K}$ becomes the product measure of $\left\vert
d^{\left\vert I\right\vert }x_{I}\right\vert _{K}$ and $\left\vert
d^{\left\vert J\right\vert }x_{J}\right\vert _{K}$. Fix $\xi_{J}^{\left(
0\right)  }\in K^{\left\vert J\right\vert }$. Then the mapping
\[%
\begin{array}
[c]{cccc}%
\boldsymbol{P}_{J,\xi_{J}^{\left(  0\right)  }}: & \mathcal{H}_{\infty}\left(
K^{n}\right)  & \rightarrow & \mathcal{H}_{\infty}\left(  K^{\left\vert
I\right\vert }\right) \\
&  &  & \\
& g\left(  x_{I},x_{J}\right)  & \rightarrow & \int\nolimits_{K^{\left\vert
J\right\vert }}\chi\left(  x_{J}\cdot\xi_{J}^{\left(  0\right)  }\right)
g\left(  x_{I},x_{J}\right)  \left\vert d^{\left\vert J\right\vert }%
x_{J}\right\vert _{K}%
\end{array}
\]
gives rise to a well-defined linear continuous operator.
\end{lemma}

\begin{proof}
By using that $g\left(  x_{I},x_{J}\right)  \in L^{1}\left(  K^{n},\left\vert
d^{n}x\right\vert _{K}\right)  $, cf. \cite[Theorem 3.15-(ii)]{Zuniga2016},
and applying Fubini's theorem, $g\left(  x_{I},x_{J}\right)  \in L^{1}\left(
K^{\left\vert J\right\vert },\left\vert d^{\left\vert J\right\vert }%
x_{J}\right\vert _{K}\right)  $ for almost all the $x_{I}$'s. Thus
\[
\mathcal{F}_{\substack{x_{I}\rightarrow\xi_{I}\\x_{J}\rightarrow\xi_{J}%
}}\left(  g\right)  {\LARGE \mid}_{\substack{\\\xi_{J}=\xi_{J}^{\left(
0\right)  }}}=\mathcal{F}_{x_{I}\rightarrow\xi_{I}}\left(  \boldsymbol{P}%
_{J,\xi_{J}^{\left(  0\right)  }}g\right)  .
\]
Now,
\begin{align*}
\left\Vert \boldsymbol{P}_{J,\xi_{J}^{\left(  0\right)  }}g\right\Vert
_{l}^{2}  &  =\int\nolimits_{K^{\left\vert I\right\vert }}\left[  \xi
_{I}\right]  _{K}^{l}\left\vert \mathcal{F}_{x_{I}\rightarrow\xi_{I}}\left(
\boldsymbol{P}_{J,\xi_{J}^{\left(  0\right)  }}g\right)  \right\vert
^{2}\left\vert d^{\left\vert I\right\vert }\xi_{I}\right\vert _{K}\\
&  =\int\nolimits_{K^{\left\vert I\right\vert }}\left[  \xi_{I}\right]
_{K}^{l}\left\vert \widehat{g}\left(  \xi_{I},\xi_{J}^{\left(  0\right)
}\right)  \right\vert ^{2}\left\vert d^{\left\vert I\right\vert }\xi
_{I}\right\vert _{K}\leq\int\nolimits_{K^{n}}\left[  \xi\right]  _{K}%
^{l}\left\vert \widehat{g}\left(  \xi_{I},\xi_{J}\right)  \right\vert
^{2}\left\vert d^{n}\xi\right\vert _{K}\\
&  =\left\Vert g\right\Vert _{l}^{2}<\infty
\end{align*}
for any $l\in\mathbb{N}$, which implies that $\boldsymbol{P}_{J,\xi
_{J}^{\left(  0\right)  }}$ is a continuous operator.
\end{proof}

\begin{lemma}
\label{lemma4C}Fix two subsets $I$, $J$ of $\left\{  1,\ldots,n\right\}  $
satisfying $I\cap J=\emptyset$ and $I\cup J=\left\{  1,\ldots,n\right\}  $.
Set $\boldsymbol{\alpha}_{I}=\left(  \alpha_{i}\right)  _{i\in I}\in
\mathbb{C}^{\left\vert I\right\vert }$ and $\boldsymbol{\beta}_{J}=\left(
\beta_{i}\right)  _{i\in J}\in\mathbb{C}^{\left\vert J\right\vert }$. Assume
that $\operatorname{Re}\left(  \alpha_{i}\right)  >0$ for $i\in I$ and
$\operatorname{Re}\left(  \beta_{i}\right)  >0$ for $i\in J$. Set for
$g\in\mathcal{H}_{\infty}$,%
\[
E_{\widehat{g}}\left(  \boldsymbol{\alpha}_{I},\boldsymbol{\beta}_{J}\right)
:=\int\limits_{K^{n}\smallsetminus\left\{  0\right\}  }\frac{\prod
\limits_{i\in I}\left\vert \xi_{i}\right\vert _{K}^{\alpha_{i}}}%
{\prod\limits_{i\in J}\left\vert \xi_{i}\right\vert _{K}^{\beta_{i}}}%
\widehat{g}\left(  \xi\right)  \left\vert d^{n}\xi\right\vert _{K},
\]
with the convention that $%
{\textstyle\prod\nolimits_{i\in\emptyset}}
\cdot=1$. Then $E_{\widehat{g}}\left(  \boldsymbol{\alpha}_{I}%
,\boldsymbol{\beta}_{J}\right)  $ gives rise to a $\mathcal{H}_{\infty}^{\ast
}$-valued function which is holomorphic in $\boldsymbol{\alpha}_{I}$ and
$\boldsymbol{\beta}_{J}$\ in the open set \ $\operatorname{Re}\left(
\alpha_{i}\right)  >0$ for $i\in I$ and $0<\operatorname{Re}\left(  \beta
_{i}\right)  <1$ for $i\in J$.
\end{lemma}

\begin{proof}
We first consider the case $J=\emptyset$,
\[
E_{\widehat{g}}\left(  \boldsymbol{\alpha}_{I}\right)  =\int\limits_{K^{n}%
\smallsetminus\left\{  0\right\}  }\prod\limits_{i\in I}\left\vert \xi
_{i}\right\vert _{K}^{\alpha_{i}}\widehat{g}\left(  \xi\right)  \left\vert
d^{n}\xi\right\vert _{K}.
\]
By Lemma \ref{lemma4B}, $E_{\widehat{g}}\left(  \boldsymbol{\alpha}%
_{I}\right)  $ defines a $\mathcal{H}_{\infty}^{\ast}$-valued function which
is holomorphic in $\boldsymbol{\alpha}_{I}$ in the open set
\ $\operatorname{Re}\left(  \alpha_{i}\right)  >0$ for $i\in I$.

We assume that $J\neq\emptyset$. \ For $L\subseteq\left\{  1,\ldots,n\right\}
$, we define%
\[
A_{L}=\left\{  \xi=\left(  \xi_{1},\ldots,\xi_{n}\right)  \in\mathbb{C}%
^{n};\left\vert \xi_{i}\right\vert _{K}>1\Leftrightarrow i\in L\right\}  .
\]
Then
\[
K^{n}=\bigsqcup\limits_{L\subseteq\left\{  1,\ldots,n\right\}  }A_{L}\text{
and }E_{\widehat{g}}\left(  \boldsymbol{\alpha}_{I},\boldsymbol{\beta}%
_{J}\right)  =\sum\limits_{L\subseteq\left\{  1,\ldots,n\right\}  }%
E_{\widehat{g}}^{\left(  L\right)  }\left(  \boldsymbol{\alpha}_{I}%
,\boldsymbol{\beta}_{J}\right)
\]
where
\[
E_{\widehat{g}}^{\left(  L\right)  }\left(  \boldsymbol{\alpha}_{I}%
,\boldsymbol{\beta}_{J}\right)  :=\int\limits_{A_{L}\smallsetminus\left\{
0\right\}  }\frac{\prod\limits_{i\in I}\left\vert \xi_{i}\right\vert
_{K}^{\alpha_{i}}}{\prod\limits_{i\in J}\left\vert \xi_{i}\right\vert
_{K}^{\beta_{i}}}\widehat{g}\left(  \xi\right)  \left\vert d^{n}\xi\right\vert
_{K}.
\]
The proof is accomplished by showing that each functional $E_{\widehat{g}%
}^{\left(  L\right)  }\left(  \boldsymbol{\alpha}_{I},\boldsymbol{\beta}%
_{J}\right)  $ satisfies the requirements announced for $E_{\widehat{g}%
}\left(  \boldsymbol{\alpha}_{I},\boldsymbol{\beta}_{J}\right)  $.

\textbf{Case 1. }If $L=\emptyset$, $A_{L}=R_{K}^{n}$.

In this case, with the notation of Lemma \ref{lemma4B_1}, we have%
\begin{multline*}
E_{\widehat{g}}^{\left(  L\right)  }\left(  \boldsymbol{\alpha}_{I}%
,\boldsymbol{\beta}_{J}\right)  =\int\limits_{R_{K}^{n}\smallsetminus\left\{
0\right\}  }\frac{\prod\limits_{i\in I}\left\vert \xi_{i}\right\vert
_{K}^{\alpha_{i}}}{\prod\limits_{i\in J}\left\vert \xi_{i}\right\vert
_{K}^{\beta_{i}}}\widehat{g}\left(  \xi\right)  \left\vert d^{n}\xi\right\vert
_{K}\\
=\int\limits_{R_{K}^{\left\vert J\right\vert }\smallsetminus\left\{
0\right\}  }\frac{1}{\prod\limits_{i\in J}\left\vert \xi_{i}\right\vert
_{K}^{\beta_{i}}}\left\{  \int\limits_{R_{K}^{\left\vert I\right\vert }}%
\prod\limits_{i\in I}\left\vert \xi_{i}\right\vert _{K}^{\alpha_{i}%
}\mathcal{F}_{x_{I}\rightarrow\xi_{I}}\left(  \boldsymbol{P}_{J,\xi_{J}%
}g\right)  \left\vert d^{\left\vert I\right\vert }\xi_{I}\right\vert
_{K}\right\}  \left\vert d^{\left\vert J\right\vert }\xi_{J}\right\vert _{K},
\end{multline*}
and%
\begin{multline*}
\left\vert E_{\widehat{g}}^{\left(  L\right)  }\left(  \boldsymbol{\alpha}%
_{I},\boldsymbol{\beta}_{J}\right)  \right\vert \leq\int\limits_{R_{K}%
^{\left\vert J\right\vert }\smallsetminus\left\{  0\right\}  }\frac{1}%
{\prod\limits_{i\in J}\left\vert \xi_{i}\right\vert _{K}^{\operatorname{Re}%
(\beta_{i})}}\times\\
\left\{  \int\limits_{R_{K}^{\left\vert I\right\vert }}\left\vert
\mathcal{F}_{x_{I}\rightarrow\xi_{I}}\left(  \boldsymbol{P}_{J,\xi_{J}%
}g\right)  \right\vert \left\vert d^{\left\vert I\right\vert }\xi
_{I}\right\vert _{K}\right\}  \left\vert d^{\left\vert J\right\vert }\xi
_{J}\right\vert _{K}\\
\leq\left\{  \int\limits_{R_{K}^{\left\vert J\right\vert }\smallsetminus
\left\{  0\right\}  }\frac{\left\vert d^{\left\vert J\right\vert }\xi
_{J}\right\vert _{K}}{\prod\limits_{i\in J}\left\vert \xi_{i}\right\vert
_{K}^{\operatorname{Re}(\beta_{i})}}\right\}  \sup_{\xi_{J}\in K^{\left\vert
J\right\vert }}\int\limits_{R_{K}^{\left\vert I\right\vert }}\left\vert
\mathcal{F}_{x_{I}\rightarrow\xi_{I}}\left(  \boldsymbol{P}_{J,\xi_{J}%
}g\right)  \right\vert \left\vert d^{\left\vert I\right\vert }\xi
_{I}\right\vert _{K}\\
\leq\left\{  \int\limits_{R_{K}^{\left\vert J\right\vert }\smallsetminus
\left\{  0\right\}  }\frac{\left\vert d^{\left\vert J\right\vert }\xi
_{J}\right\vert _{K}}{\prod\limits_{i\in J}\left\vert \xi_{i}\right\vert
_{K}^{\operatorname{Re}(\beta_{i})}}\right\}  \sup_{\xi_{J}\in K^{\left\vert
J\right\vert }}\int\limits_{R_{K}^{\left\vert I\right\vert }}\frac{1}{\left[
\xi_{I}\right]  _{K}^{\frac{l}{2}}}\left[  \xi_{I}\right]  _{K}^{\frac{l}{2}%
}\left\vert \mathcal{F}_{x_{I}\rightarrow\xi_{I}}\left(  \boldsymbol{P}%
_{J,\xi_{J}}g\right)  \right\vert \left\vert d^{\left\vert I\right\vert }%
\xi_{I}\right\vert _{K}.
\end{multline*}
Now, by applying Cauchy-Schwarz inequality and Lemma \ref{lemma4B_1},
\begin{gather*}
\left\vert E_{\widehat{g}}^{\left(  L\right)  }\left(  \boldsymbol{\alpha}%
_{I},\boldsymbol{\beta}_{J}\right)  \right\vert \leq\left\{  \int
\limits_{R_{K}^{\left\vert J\right\vert }\smallsetminus\left\{  0\right\}
}\frac{\left\vert d^{\left\vert J\right\vert }\xi_{J}\right\vert _{K}}%
{\prod\limits_{i\in J}\left\vert \xi_{i}\right\vert _{K}^{\operatorname{Re}%
(\beta_{i})}}\right\}  \times\\
\sup_{\xi_{J}\in K^{\left\vert J\right\vert }}\sqrt{\int\limits_{R_{K}%
^{\left\vert I\right\vert }}\frac{\left\vert d^{\left\vert I\right\vert }%
\xi_{I}\right\vert _{K}}{\left[  \xi_{I}\right]  _{K}^{l}}}\sqrt
{\int\limits_{R_{K}^{\left\vert I\right\vert }}\left[  \xi_{I}\right]
_{K}^{l}\left\vert \mathcal{F}\left(  \boldsymbol{P}_{J,\xi_{J}}g\right)
\right\vert ^{2}\left\vert d^{\left\vert I\right\vert }\xi_{I}\right\vert
_{K}}\\
\leq\left\{  \int\limits_{R_{K}^{\left\vert J\right\vert }\smallsetminus
\left\{  0\right\}  }\frac{\left\vert d^{\left\vert J\right\vert }\xi
_{J}\right\vert _{K}}{\prod\limits_{i\in J}\left\vert \xi_{i}\right\vert
_{K}^{\operatorname{Re}(\beta_{i})}}\right\}  \sup_{\xi_{J}\in K^{\left\vert
J\right\vert }}\sqrt{\int\limits_{K^{n}}\left[  \xi\right]  _{K}^{l}\left\vert
\widehat{g}\left(  \xi\right)  \right\vert ^{2}\left\vert d^{n}\xi\right\vert
_{K}}\\
\leq\left\{  \int\limits_{R_{K}^{\left\vert J\right\vert }\smallsetminus
\left\{  0\right\}  }\frac{\left\vert d^{\left\vert J\right\vert }\xi
_{J}\right\vert _{K}}{\prod\limits_{i\in J}\left\vert \xi_{i}\right\vert
_{K}^{\operatorname{Re}(\beta_{i})}}\right\}  \left\Vert g\right\Vert _{l}.
\end{gather*}
We now use that, if $0<\operatorname{Re}\left(  \beta_{i}\right)  <1$ then
$\frac{1}{\left\vert \xi_{i}\right\vert _{K}^{\operatorname{Re}\left(
\beta_{i}\right)  }}\in L^{1}\left(  R_{K},\left\vert d\xi_{i}\right\vert
_{K}\right)  $, to conclude that
\begin{equation}
\left\vert E_{\widehat{g}}^{\left(  L\right)  }\left(  \boldsymbol{\alpha}%
_{I},\boldsymbol{\beta}_{J}\right)  \right\vert \leq C(\boldsymbol{\beta}%
_{J})\left\Vert g\right\Vert _{l} \label{Eq2}%
\end{equation}
for any $l\in\mathbb{N}$. This implies that $E_{\widehat{g}}^{\left(
L\right)  }\left(  \boldsymbol{\alpha}_{I},\boldsymbol{\beta}_{J}\right)  $ is
a $\mathcal{H}_{\infty}^{\ast}$-valued function for $\boldsymbol{\beta}_{J}$
in the set $0<\operatorname{Re}\left(  \beta_{i}\right)  <1$ for $i\in J$ and
$\operatorname{Re}\left(  \alpha_{i}\right)  >0$ for $i\in I$. In order to
show that $E_{\widehat{g}}^{\left(  L\right)  }\left(  \boldsymbol{\alpha}%
_{I},\boldsymbol{\beta}_{J}\right)  $ is holomorphic in $\left(
\boldsymbol{\alpha}_{I},\boldsymbol{\beta}_{J}\right)  $ in
$\boldsymbol{\alpha}_{I},\boldsymbol{\beta}_{J}$ in the set
$0<\operatorname{Re}\left(  \beta_{i}\right)  <1$ for $i\in J$ and
$\operatorname{Re}\left(  \alpha_{i}\right)  >0$ for $i\in I$, we show that
$E_{\widehat{g}}^{\left(  L\right)  }\left(  \boldsymbol{\alpha}%
_{I},\boldsymbol{\beta}_{J}\right)  $ is holomorphic in each variable
separately. This fact is established by using (\ref{Eq2}) and Lemma 5.3.1 in
\cite{Igusa}.

\textbf{Case 2. }If $L=\left\{  1,\ldots,n\right\}  $, $A_{L}=\left\{  \xi\in
K^{n};\left\vert \xi_{i}\right\vert _{K}>1\text{ for }i=1,\ldots,n\right\}  $.

In this case%
\begin{align*}
\left\vert E_{\widehat{g}}^{\left(  L\right)  }\left(  \boldsymbol{\alpha}%
_{I},\boldsymbol{\beta}_{J}\right)  \right\vert  &  \leq\int\limits_{A_{L}%
}\prod\limits_{i\in I}\left\vert \xi_{i}\right\vert _{K}^{\operatorname{Re}%
\left(  \alpha_{i}\right)  }\left\vert \widehat{g}\left(  \xi\right)
\right\vert \left\vert d^{n}\xi\right\vert _{K}\\
&  \leq\int\limits_{A_{L}}\left[  \xi\right]  _{K}^{\sum_{i\in I}\left\lceil
\operatorname{Re}\left(  \alpha_{i}\right)  \right\rceil }\left\vert
\widehat{g}\left(  \xi\right)  \right\vert \left\vert d^{n}\xi\right\vert
_{K}\leq C(l,n)\left\Vert g\right\Vert _{l+2\sum_{i\in I}\left\lceil
\operatorname{Re}\left(  \alpha_{i}\right)  \right\rceil ,}%
\end{align*}
for any positive integer $l>n$, which implies that $E_{\widehat{g}}^{\left(
L\right)  }\left(  \boldsymbol{\alpha}_{I}\right)  $ is a $\mathcal{H}%
_{\infty}^{\ast}$-valued function in the set $\operatorname{Re}\left(
\alpha_{i}\right)  >0$ for $i\in I$.

We now analyze the case where $L$ is a non-empty and proper subset from
$\left\{  1,\ldots,n\right\}  $ and $J\neq\emptyset$.

\textbf{Case 3. }If $J\cap L=\emptyset$, i.e. $L\subseteq I$.

In this case, by proceeding as in the proof of Case 1, and using that
\[
\int_{K^{\left\vert L\right\vert }}\frac{\left\vert d^{\left\vert L\right\vert
}\xi_{L}\right\vert _{K}}{\left[  \xi\right]  _{K}^{l}}<\infty\text{ for
}l>n\text{,}%
\]
one gets that%
\[
\left\vert E_{\widehat{g}}^{\left(  L\right)  }\left(  \boldsymbol{\alpha}%
_{I},\boldsymbol{\beta}_{J}\right)  \right\vert \leq C(\boldsymbol{\beta}%
_{J},l,n)\left\Vert g\right\Vert _{l+2\sum_{i\in I}\left\lceil
\operatorname{Re}(\alpha_{i})\right\rceil },
\]
which implies that $E_{\widehat{g}}^{\left(  L\right)  }\left(
\boldsymbol{\alpha}_{I},\boldsymbol{\beta}_{J}\right)  $ is a $\mathcal{H}%
_{\infty}^{\ast}$-valued function for $\boldsymbol{\beta}_{J}$ in the set
$0<\operatorname{Re}\left(  \beta_{i}\right)  <1$ for $i\in J$ and
$\operatorname{Re}\left(  \alpha_{i}\right)  >0$ for $i\in I$. The
verification that $E_{\widehat{g}}^{\left(  L\right)  }\left(
\boldsymbol{\alpha}_{I},\boldsymbol{\beta}_{J}\right)  $ is holomorphic in
$\boldsymbol{\alpha}_{I},\boldsymbol{\beta}_{J}$ in the set
$0<\operatorname{Re}\left(  \beta_{i}\right)  <1$ for $i\in J$ and
$\operatorname{Re}\left(  \alpha_{i}\right)  >0$ for $i\in I$ is done like in
\ Case 1.

\textbf{Case 4. } $J\cap L\neq\emptyset$ and $I\cap L=\emptyset$ (i.e.
$L\subseteq J$) and $M:=J\smallsetminus L$.

In this case, taking $\left\{  1,\ldots,n\right\}  =L%
{\textstyle\bigsqcup}
L^{\prime}$ and
\begin{equation}
A_{L}\smallsetminus\left\{  0\right\}  =\left(  R_{K}^{\left\vert L^{\prime
}\right\vert }\smallsetminus\left\{  0\right\}  \right)
{\textstyle\bigsqcup}
\left(  K^{\left\vert L\right\vert }\smallsetminus\left\{  0\right\}  \right)
, \label{partition}%
\end{equation}
and using the notation of Lemma \ref{lemma4B_1}, $E_{\widehat{g}}^{\left(
L\right)  }\left(  \boldsymbol{\alpha}_{I},\boldsymbol{\beta}_{J}\right)
$\ equals
\[
\int\limits_{R_{K}^{\left\vert L^{\prime}\right\vert }\smallsetminus\left\{
0\right\}  }\frac{1}{\prod\limits_{i\in M}\left\vert \xi_{i}\right\vert
_{K}^{\beta_{i}}}\left\{  \int\limits_{K^{\left\vert L\right\vert
}\smallsetminus\left\{  0\right\}  }\frac{\prod\limits_{i\in I}\left\vert
\xi_{i}\right\vert _{K}^{\alpha_{i}}}{\prod\limits_{i\in L}\left\vert \xi
_{i}\right\vert _{K}^{\beta_{i}}}\mathcal{F}_{x_{L}\rightarrow\xi_{L}}\left(
\boldsymbol{P}_{L^{\prime},\xi_{L^{\prime}}}g\right)  \left\vert d^{\left\vert
L\right\vert }\xi_{L}\right\vert _{K}\right\}  \left\vert d^{\left\vert
L^{\prime}\right\vert }\xi_{L^{\prime}}\right\vert _{K}%
\]
and thus
\begin{multline*}
\left\vert E_{\widehat{g}}^{\left(  L\right)  }\left(  \boldsymbol{\alpha}%
_{I},\boldsymbol{\beta}_{J}\right)  \right\vert \leq\\
\left\{  \int\limits_{R_{K}^{\left\vert L^{\prime}\right\vert }\smallsetminus
\left\{  0\right\}  }\frac{\left\vert d^{\left\vert L^{\prime}\right\vert }%
\xi_{L^{\prime}}\right\vert _{K}}{\prod\limits_{i\in M}\left\vert \xi
_{i}\right\vert _{K}^{\operatorname{Re}(\beta_{i})}}\right\}  \sup
_{\xi_{L^{\prime}}\in K^{\left\vert L^{\prime}\right\vert }}\int
\limits_{K^{\left\vert L\right\vert }}\left\vert \mathcal{F}_{x_{L}%
\rightarrow\xi_{L}}\left(  \boldsymbol{P}_{L^{\prime},\xi_{L^{\prime}}%
}g\right)  \right\vert \left\vert d^{\left\vert L\right\vert }\xi
_{L}\right\vert _{K}\\
=\left\{  \int\limits_{R_{K}^{\left\vert M\right\vert }\smallsetminus\left\{
0\right\}  }\frac{\left\vert d^{\left\vert M\right\vert }\xi_{M}\right\vert
_{K}}{\prod\limits_{i\in M}\left\vert \xi_{i}\right\vert _{K}%
^{\operatorname{Re}(\beta_{i})}}\right\}  \times\\
\sup_{\xi_{L^{\prime}}\in K^{\left\vert L^{\prime}\right\vert }}%
\int\limits_{K^{\left\vert L\right\vert }}\frac{1}{\left[  \xi_{_{L}}\right]
_{K}^{\frac{l}{2}}}\left[  \xi_{L}\right]  _{K}^{\frac{l}{2}}\left\vert
\mathcal{F}_{x_{L}\rightarrow\xi_{L}}\left(  \boldsymbol{P}_{L^{\prime}%
,\xi_{L^{\prime}}}g\right)  \right\vert \left\vert d^{\left\vert L\right\vert
}\xi_{L}\right\vert _{K}.
\end{multline*}
By taking $l>n$, $0<\operatorname{Re}(\beta_{i})<1$ for $i\in M$, and applying
Cauchy-Schwarz inequality and Lemma \ref{lemma4B_1},%
\begin{multline*}
\left\vert E_{\widehat{g}}^{\left(  L\right)  }\left(  \boldsymbol{\alpha}%
_{I},\boldsymbol{\beta}_{J}\right)  \right\vert \leq C\left(
\boldsymbol{\beta}_{J},l,n\right)  \times\\
\sup_{\xi_{L^{\prime}}\in K^{\left\vert L^{\prime}\right\vert }}\sqrt
{\int\limits_{K^{\left\vert L\right\vert }}\left[  \xi_{L}\right]  _{K}%
^{l}\left\vert \mathcal{F}_{x_{L}\rightarrow\xi_{L}}\left(  \boldsymbol{P}%
_{L^{\prime},\xi_{L^{\prime}}}g\right)  \right\vert ^{2}\left\vert
d^{\left\vert L\right\vert }\xi_{L}\right\vert _{K}}\leq C\left(
\boldsymbol{\beta}_{J},l,n\right)  \left\Vert g\right\Vert _{l}.
\end{multline*}

\textbf{Case 5. } $J\cap L\neq\emptyset$ , $I\cap L\neq\emptyset$, and
$M:=J\smallsetminus L$.

In this case, taking $\left\{  1,\ldots,n\right\}  =L%
{\textstyle\bigsqcup}
L^{\prime}$ and using (\ref{partition}), and proceeding like in Case 4,%
\begin{multline*}
\left\vert E_{\widehat{g}}^{\left(  L\right)  }\left(  \boldsymbol{\alpha}%
_{I},\boldsymbol{\beta}_{J}\right)  \right\vert \leq\left\{  \int
\limits_{R_{K}^{\left\vert M\right\vert }\smallsetminus\left\{  0\right\}
}\frac{\left\vert d^{\left\vert M\right\vert }\xi_{M}\right\vert _{K}}%
{\prod\limits_{i\in M}\left\vert \xi_{i}\right\vert _{K}^{\operatorname{Re}%
(\beta_{i})}}\right\}  \times\\
\sup_{\xi_{L^{\prime}}\in K^{\left\vert L^{\prime}\right\vert }}%
\int\limits_{K^{\left\vert L\right\vert }}\text{ }\prod\limits_{i\in I\cap
L}\left\vert \xi_{i}\right\vert _{K}^{\operatorname{Re}(\alpha_{i})}\left\vert
\mathcal{F}_{x_{L^{\prime}}\rightarrow\xi_{L^{\prime}}}\left(  \boldsymbol{P}%
_{L^{\prime},\xi_{L^{\prime}}}g\right)  \right\vert \left\vert d^{\left\vert
L\right\vert }\xi_{L}\right\vert _{K}\\
\leq C\left(  \boldsymbol{\beta}_{J},l,n\right)  \left\Vert g\right\Vert
_{l+2\sum_{i\in I\cap L}\left\lceil \operatorname{Re}(\alpha_{i})\right\rceil
}%
\end{multline*}
for $l>n$.
\end{proof}

\section{Local zeta functions in $\mathcal{H}_{\infty}$}

We fix a non-constant polynomial $\mathfrak{f}$ in $R_{K}\left[  \xi
_{1},\ldots,\xi_{n}\right]  $ of degree $d$. We set $\widetilde{\mathfrak{f}%
}\left(  \xi\right)  :=\mathfrak{f}\left(  -\xi\right)  $. Then $\widehat
{\left\vert \widetilde{\mathfrak{f}}\right\vert _{K}^{\overline{s}}}$,
$\operatorname{Re}(s)>0$, \ defines a distribution from $\mathcal{D}^{\prime}$
satisfying $\widehat{\widehat{\left\vert \widetilde{\mathfrak{f}}\right\vert
_{K}^{\overline{s}}}}=\left\vert \mathfrak{f}\right\vert _{K}^{\overline{s}}$
in $\mathcal{D}^{\prime}$ for $\operatorname{Re}(s)>0$. In addition,
$\widehat{\left\vert \widetilde{\mathfrak{f}}\right\vert _{K}^{\overline{s}}}%
$, with $\operatorname{Re}(s)>0$, gives rise to a holomorphic $\mathcal{H}%
_{\infty}^{\ast}$-valued function in $s$, cf. Lemma \ref{lemma4B}. In this
section we establish the existence of a meromorphic continuation of
$\mathcal{H}_{\infty}^{\ast}$-valued functions of type%
\begin{equation}
Z_{\widehat{g}}\left(  s,\mathfrak{f}\right)  :=\left[  \widehat{\left\vert
\widetilde{\mathfrak{f}}\right\vert _{K}^{\overline{s}}},g\right]
=\int\limits_{K^{n}\smallsetminus\widetilde{\mathfrak{f}}^{-1}\left(
0\right)  }\left\vert \mathfrak{f}\left(  \xi\right)  \right\vert _{K}%
^{s}\widehat{g}\left(  \xi\right)  \left\vert d^{n}\xi\right\vert _{K}\text{,
} \label{Zeta_function_1}%
\end{equation}
$\operatorname{Re}(s)>0$, $g\in\mathcal{H}_{\infty}\left(  K^{n}\right)  $, to
the whole complex plane. Since $\mathcal{D}\left(  K^{n}\right)
\subset\mathcal{H}_{\infty}\left(  K^{n}\right)  $, integrals of type
(\ref{Zeta_function_1}) are generalizations of the classical Igusa's local
zeta functions, see e.g. \cite{Igusa}, \cite{Igusa1}. Before further
discussion we present an example that illustrates the analogies an differences
between the classical Igusa's zeta functions and the local zeta functions on
$\mathcal{H}_{\infty}$.

\subsection{\label{Sect_Local_zeta_forms}Local zeta functions for strongly
non-degenerate forms modulo $\pi$}

\begin{notation}
We denote by `$\overline{\cdot}$', the reduction modulo $\pi$, i.e. the
canonical mapping $R_{K}^{n}\rightarrow\left(  R_{K}/\pi R_{K}\right)
^{n}=\mathbb{F}_{q}^{n}$. If $\mathfrak{f}$ is a polynomial with coefficients
in $R_{K}$, we denote by $\overline{\mathfrak{f}}$, the polynomial obtained by
reducing modulo $\pi$ the coefficients of $\mathfrak{f}$.
\end{notation}

We take $g\left(  x\right)  =\mathcal{F}_{\xi\rightarrow x}^{-1}\left(
e^{-\left\Vert \xi\right\Vert _{K}^{\alpha}}\right)  $, with $\alpha>0$. By
Lemma \ref{Lemma2A}-(iii), $g\in\mathcal{H}_{\infty}\left(  K^{n}\right)  $.
It is interesting to mention that function $\mathcal{F}_{\xi\rightarrow
x}^{-1}\left(  e^{-t\left\Vert \xi\right\Vert _{K}^{\alpha}}\right)  $ with
$t>0$, $\alpha>0$ is the `fundamental solution' of the heat equation over
$K^{n}$, see e.g. \cite[Section 2.2.7]{Zuniga-LNM-2016}.\ We also pick a
homogeneous polynomial $\mathfrak{f}$ with coefficients in $R_{K}%
\smallsetminus\pi R_{K}$ of degree $d$, satisfying $\overline{\mathfrak{f}%
}\left(  \overline{a}\right)  =\nabla\overline{\mathfrak{f}}\left(
\overline{a}\right)  =\overline{0}$ implies $\overline{a}=\overline{0}$. We
set
\[
\left[  \widehat{\left\vert \widetilde{\mathfrak{f}}\right\vert _{K}%
^{\overline{s}}},g\right]  =%
{\displaystyle\int\limits_{K^{n}\smallsetminus\mathfrak{f}^{-1}\left(
0\right)  }}
\left\vert \mathfrak{f}\left(  \xi\right)  \right\vert _{K}^{s}e^{-\left\Vert
\xi\right\Vert _{K}^{\alpha}}\left\vert d^{n}\xi\right\vert _{K}\text{ for
}\operatorname{Re}(s)>0.
\]

\textbf{Claim.} $\left[  \widehat{\left\vert \widetilde{\mathfrak{f}%
}\right\vert _{K}^{\overline{s}}},g\right]  $ admits meromorphic continuation
to the whole \ complex plane and the real parts of the possible poles belong
to the set $\left\{  -1\right\}  \cup\cup_{l\in\mathbb{N}}\left\{
\frac{-\left(  n+\alpha l\right)  }{d}\right\}  $.

To establish the Claim we proceed as follows. We use the partition
$K^{n}\smallsetminus\left\{  0\right\}  =%
{\textstyle\bigsqcup\nolimits_{j=-\infty}^{\infty}}
\pi^{j}S_{0}^{n}$\ to obtain%
\begin{gather}
\left[  \widehat{\left\vert \widetilde{\mathfrak{f}}\right\vert _{K}%
^{\overline{s}}},g\right]  =%
{\displaystyle\sum\limits_{j=-\infty}^{\infty}}
\text{ }%
{\displaystyle\int\limits_{\pi^{j}S_{0}^{n}}}
\left\vert \mathfrak{f}\left(  \xi\right)  \right\vert _{K}^{s}e^{-\left\Vert
\xi\right\Vert _{K}^{\alpha}}\left\vert d^{n}\xi\right\vert _{K}%
\label{Eq_Z_0}\\
=\left(
{\displaystyle\int\limits_{S_{0}^{n}}}
\left\vert \mathfrak{f}\left(  \xi\right)  \right\vert _{K}^{s}\left\vert
d^{n}\xi\right\vert _{K}\right)  \left(
{\displaystyle\sum\limits_{j=-\infty}^{\infty}}
q^{-jds-jn}e^{-q^{-j\alpha}}\right) \nonumber\\
=\left(
{\displaystyle\int\limits_{S_{0}^{n}}}
\left\vert \mathfrak{f}\left(  \xi\right)  \right\vert _{K}^{s}\left\vert
d^{n}\xi\right\vert _{K}\right)  \left(  \frac{1}{1-q^{-n}}%
{\displaystyle\int\limits_{K^{n}}}
\left\Vert \xi\right\Vert _{K}^{ds}e^{-\left\Vert \xi\right\Vert _{K}^{\alpha
}}\left\vert d^{n}\xi\right\vert _{K}\right)  =:Z_{0}(s)Z_{1}(s).\nonumber
\end{gather}
By using that
\[
Z(s,\mathfrak{f}):=%
{\displaystyle\int\limits_{R_{K}^{n}}}
\left\vert \mathfrak{f}\left(  \xi\right)  \right\vert _{K}^{s}\left\vert
d^{n}\xi\right\vert _{K}\text{ }\frac{L\left(  q^{-s}\right)  }{\left(
1-q^{-1-s}\right)  \left(  1-q^{-n-ds}\right)  }\text{ for }\operatorname{Re}%
(s)>0\text{,}%
\]
where $L\left(  q^{-s}\right)  $\ is a polynomial in $q^{-s}$, see e.g.
\cite[Proposition 10.2.1]{Igusa}, and the partition $R_{K}^{n}=\pi R_{K}^{n}%
{\textstyle\bigsqcup}
S_{0}^{n}$, we have $Z(s,\mathfrak{f})=q^{-n-ds}Z(s,\mathfrak{f})+Z_{0}(s)$,
and consequently
\begin{equation}
Z_{0}(s)=\frac{L\left(  q^{-s}\right)  }{\left(  1-q^{-1-s}\right)  }.
\label{Eq_Z_1}%
\end{equation}
Now%
\begin{align}
Z_{1}(s)  &  =\frac{1}{1-q^{-n}}%
{\displaystyle\int\limits_{R_{K}^{n}}}
\left\Vert \xi\right\Vert _{K}^{ds}e^{-\left\Vert \xi\right\Vert _{K}^{\alpha
}}\left\vert d^{n}\xi\right\vert _{K}+\frac{1}{1-q^{-n}}%
{\displaystyle\int\limits_{K^{n}\smallsetminus R_{K}^{n}}}
\left\Vert \xi\right\Vert _{K}^{ds}e^{-\left\Vert \xi\right\Vert _{K}^{\alpha
}}\left\vert d^{n}\xi\right\vert _{K}\label{Eq_Z_2}\\
&  =:Z_{1,1}(s)+Z_{1,2}(s).\nonumber
\end{align}
The integral $Z_{1,2}(s)$ is holomorphic in the whole complex plane. To study
$Z_{1,1}(s)$ we use $e^{-\left\Vert \xi\right\Vert _{K}^{\alpha}}=\sum
_{l=0}^{L}\frac{\left(  -1\right)  ^{l}}{l!}\left\Vert \xi\right\Vert
_{K}^{\alpha l}+f\left(  \left\Vert \xi\right\Vert _{K}\right)  $ as follows:%
\begin{align}
Z_{1,1}(s)  &  =\frac{1}{1-q^{-n}}\sum_{l=0}^{L}\frac{\left(  -1\right)  ^{l}%
}{l!}%
{\displaystyle\int\limits_{R_{K}^{n}}}
\left\Vert \xi\right\Vert _{K}^{ds+\alpha l}\left\vert d^{n}\xi\right\vert
_{K}+\frac{1}{1-q^{-n}}%
{\displaystyle\int\limits_{R_{K}^{n}}}
\left\Vert \xi\right\Vert _{K}^{ds}f\left(  \left\Vert \xi\right\Vert
_{K}\right)  \left\vert d^{n}\xi\right\vert _{K}\label{Eq_Z_3}\\
&  =\sum_{l=0}^{L}\frac{\left(  -1\right)  ^{l}}{l!}\frac{1}{1-q^{-ds-n-\alpha
l}}+\frac{1}{1-q^{-n}}%
{\displaystyle\int\limits_{R_{K}^{n}}}
\left\Vert \xi\right\Vert _{K}^{ds}f\left(  \left\Vert \xi\right\Vert
_{K}\right)  \left\vert d^{n}\xi\right\vert _{K}.\nonumber
\end{align}
The announced Claim follows from (\ref{Eq_Z_0})-(\ref{Eq_Z_3}). The local zeta
functions $\left[  \widehat{\left\vert \widetilde{\mathfrak{f}}\right\vert
_{K}^{\overline{s}}},g\right]  $ admit meromorphic continuations to the whole
complex plane but they are not rational functions of $q^{-s}$ and the real
parts of the possible poles are `real' negative numbers. In addition, if take
$\alpha=1$, then the real parts of the poles of the meromorphic continuations
resemble the poles of Archimedean zeta functions, see e.g. \cite[Theorem
5.4.1]{Igusa}.

\subsection{Multidimensional Vladimirov operators}

\subsubsection{Riesz kernels}

For $\alpha\in\mathbb{C}$, we set $\Gamma\left(  \alpha\right)  :=\frac
{1-q^{\alpha-1}}{1-q^{-\alpha}}$. The function%
\[
f_{\alpha}\left(  x\right)  =\frac{\left\vert x\right\vert _{K}^{\alpha-1}%
}{\Gamma\left(  \alpha\right)  }\text{ for }\alpha\neq\mu_{j}\text{, }%
\alpha\neq1+\mu_{j}\text{, }j\in\mathbb{Z}\text{, }x\in K\text{,}%
\]
where $\mu_{j}=\frac{2\pi\sqrt{-1}j}{\ln q}$, $j\in\mathbb{Z}$, gives rise to
a distribution from $\mathcal{D}^{\prime}(K)$\ called \textit{the
one-dimensional} \textit{ Riesz kernel}. This distribution has a meromorphic
extension to the whole complex plane $\alpha$, with poles at the points
$\alpha=\mu_{j}$, $1+\mu_{j}$, $j\in\mathbb{Z}$. The distribution
$f_{0}\left(  x\right)  $ is defined by taking%
\[
f_{0}\left(  x\right)  =\lim_{\alpha\rightarrow0}f_{\alpha}\left(  x\right)
=\delta\left(  x\right)  \text{, }x\in K\text{ (the Dirac distribution)}%
\]
where the limit is understood in the weak sense. Notice that $f_{0}\left(
x\right)  =\lim_{\alpha\rightarrow\mu_{j}}f_{\alpha}\left(  x\right)  $ for
$j\in\mathbb{Z}$. The definition of the distribution $f_{1}\left(  x\right)  $
requires to substitute the space of test functions by the Lizorkin space of
the first kind, see e.g. \cite[Section 9.2]{A-K-S}. We do not use this
approach in this article.

We recall that if $\alpha\in\mathbb{C}$, with $\operatorname{Re}(\alpha)\neq
1$, then
\begin{equation}%
{\displaystyle\int\nolimits_{K}}
f_{\alpha}\left(  \xi\right)  \widehat{\phi}\left(  \xi\right)  \left\vert
d\xi\right\vert _{K}=%
{\displaystyle\int\nolimits_{K}}
\left\vert x\right\vert _{K}^{-\alpha}\phi\left(  x\right)  \left\vert
dx\right\vert _{K}, \label{Eq18}%
\end{equation}
for $\phi\in\mathcal{D}(K)$ with the convention that $f_{0}\left(  \xi\right)
=\delta\left(  \xi\right)  $, see e.g. \cite[Theorem 4.5]{Taibleson}.

\begin{remark}
\label{Nota3}If $T\in\mathcal{D}^{\prime}(K^{n})$ and $G\in\mathcal{D}%
^{\prime}(K^{m})$, then its \textit{direct product} is the distribution
defined by the formula%
\[
\left(  T\left(  x\right)  \times G\left(  y\right)  ,\varphi\right)  =\left(
T\left(  x\right)  ,\left(  G\left(  y\right)  ,\varphi\left(  x,y\right)
\right)  \right)  \text{ for }\varphi\left(  x,y\right)  \in\mathcal{D}%
(K^{n+m}).
\]
By using that any test function $\varphi\left(  x,y\right)  $ in
$\mathcal{D}(K^{n+m})$ is a linear combination\ of test functions of the form
$\phi_{k}\left(  x\right)  \psi_{k}\left(  y\right)  $ with $\phi_{k}\left(
x\right)  \in\mathcal{D}(K^{n})$ and $\psi_{k}\left(  y\right)  \in
\mathcal{D}(K^{m})$, one verifies that the Fourier transform of $T\times G$ in
$\mathcal{D}^{\prime}(K^{n+m})$ is given by the formula%
\[
\left(  \widehat{T\times G},\varphi\right)  =\left(  \widehat{T}\times
\widehat{G},\varphi\right)  =\left(  T\times G,\widehat{\varphi}\right)
\text{ for }\varphi\in\mathcal{D}(K^{n+m}).
\]
For further details, the reader may consult \cite[Chap. 1, Sect. VI]{V-V-Z}.
\end{remark}

Let $\boldsymbol{\alpha}=\left(  \alpha_{1},\ldots,\alpha_{n}\right)
\in\mathbb{C}^{n}$ and $x=\left(  x_{1},\ldots,x_{n}\right)  \in K^{n}$. We
define $f_{\boldsymbol{\alpha}}\left(  x\right)  :=f_{\alpha_{1}}\left(
x_{_{1}}\right)  \times\cdots\times f_{\alpha_{n}}\left(  x_{n}\right)  \in$
$\mathcal{D}^{\prime}\left(  K^{n}\right)  $ as \textit{the multidimensional
Riesz kernel}, which is the direct product of the unidimensional Riesz kernels
$f_{\alpha_{j}}\left(  x_{j}\right)  $, $j=1,\ldots,n$. We will identify the
distribution $f_{\boldsymbol{\alpha}}\left(  x\right)  $ with the function $%
{\textstyle\prod\nolimits_{i=1}^{n}}
f_{\alpha_{i}}\left(  x_{_{i}}\right)  $. Thus, from Remark \ref{Nota3} and
(\ref{Eq18}), for $\operatorname{Re}(\alpha_{j})\neq1$ for $j\in\left\{
1,\ldots,n\right\}  $,
\begin{equation}
\mathcal{F}_{x\rightarrow\xi}\left(  f_{\boldsymbol{\alpha}}\left(  x\right)
\right)  =%
{\textstyle\prod\nolimits_{i=1}^{n}}
\left\vert \xi_{i}\right\vert _{K}^{-\alpha_{i}}\text{ in }\mathcal{D}%
^{\prime}\left(  K^{n}\right)  , \label{Eq4}%
\end{equation}
with the convention $f_{\boldsymbol{0}}\left(  x\right)  =\delta\left(
x\right)  \in\mathcal{D}^{\prime}\left(  K^{n}\right)  $.

\subsubsection{Vladimirov operators}

Let $\boldsymbol{\alpha}=\left(  \alpha_{1},\ldots,\alpha_{n}\right)
\in\mathbb{C}^{n}$, with $\operatorname{Re}(\alpha_{i})\geq0$ for all $i$.
\textit{The multidimensional Vladimirov\ operator} is defined as
\begin{equation}
\left(  \boldsymbol{D}^{\boldsymbol{\alpha}}\phi\right)  \left(  x\right)
=\mathcal{F}_{\xi\rightarrow x}^{-1}\left(
{\textstyle\prod\nolimits_{i=1}^{n}}
\left\vert \xi_{i}\right\vert _{K}^{\alpha_{i}}\text{ }\mathcal{F}%
_{x\rightarrow\xi}\phi\right)  \text{, }\boldsymbol{\alpha}\in\mathbb{C}%
^{n}\text{, }\phi\in\mathcal{D}\left(  K^{n}\right)  \text{.}
\label{Vladimirov_operator}%
\end{equation}
Notice that $\boldsymbol{D}^{\boldsymbol{\alpha}}:\mathcal{H}_{\infty}\left(
K^{n}\right)  \rightarrow\mathcal{H}_{\infty}\left(  K^{n}\right)  $ is
continuous operator, cf. Lemma \ref{lemma3A}. Let $\left[  \cdot,\cdot\right]
$ denote the pairing between $\mathcal{H}_{\infty}^{\ast}$\ and $\mathcal{H}%
_{\infty}$, see (\ref{pairing}). Then $\boldsymbol{D}^{\boldsymbol{\alpha}}$
has an adjoint operator $\boldsymbol{D}^{\boldsymbol{\alpha}\ast}$:
$\mathcal{H}_{\infty}^{\ast}\left(  K^{n}\right)  \rightarrow\mathcal{H}%
_{\infty}^{\ast}\left(  K^{n}\right)  $. Note that
\[
\left(  \boldsymbol{D}^{\boldsymbol{\alpha}}\phi\right)  \left(  x\right)
=f_{-\boldsymbol{\alpha}}\left(  x\right)  \ast\phi\left(  x\right)  =\left(
f_{\alpha_{1}}\left(  y_{1}\right)  \times\cdots\times f_{\alpha_{n}}\left(
y_{n}\right)  ,\phi\left(  x-y\right)  \right)  ,
\]
for $\phi\in\mathcal{D}\left(  K^{n}\right)  $, in the case $\operatorname{Re}%
(\alpha_{j})\in\left[  0,1\right)  \cup\left(  1,\infty\right)  $ for
$j=1,\ldots,n$.

\begin{proposition}
\label{Prop3}(i) Let $\boldsymbol{\alpha}=\left(  \alpha_{1},\ldots,\alpha
_{n}\right)  \in\mathbb{C}^{n}$, with $\operatorname{Re}(\alpha_{i})>0$ for
all $i$, $\boldsymbol{\beta}=\left(  \beta_{1},\ldots,\beta_{n}\right)
\in\mathbb{C}^{n}$, with $\operatorname{Re}(\beta_{i})>0$ for all $i$. The
following formula holds:
\[
\left[  \boldsymbol{D}^{\boldsymbol{\beta}\ast}\mathcal{F}\left\{
{\textstyle\prod\nolimits_{i=1}^{n}}
\left\vert x_{i}\right\vert _{K}^{\overline{\alpha}_{i}-1}\right\}  ,g\right]
=\left[  \mathcal{F}\left\{
{\textstyle\prod\nolimits_{i=1}^{n}}
\left\vert x_{i}\right\vert _{K}^{\overline{\alpha}_{i}+\overline{\beta}%
_{i}-1}\right\}  ,g\right]
\]
for $g\in\mathcal{H}_{\infty}\left(  K^{n}\right)  $ and $\mathcal{F}\left\{
{\textstyle\prod\nolimits_{i=1}^{n}}
\left\vert x_{i}\right\vert _{K}^{\overline{\alpha}_{i}-1}\right\}
\in\mathcal{H}_{\infty}^{\ast}\left(  K^{n}\right)  $. In particular, the
distribution $\mathcal{F}_{y\rightarrow x}\left[
{\textstyle\prod\nolimits_{i=1}^{n}}
\left\vert y_{i}\right\vert _{K}^{-\alpha_{i}}\right]  $ is a $\mathcal{H}%
_{\infty}^{\ast}$-valued function, which admits an analytic continuation to
$\mathbb{C}^{n}$.

\noindent(ii) Let $\boldsymbol{\alpha}=\left(  \alpha_{1},\ldots,\alpha
_{n}\right)  \in\mathbb{C}^{n}$, $\boldsymbol{\beta}=\left(  \beta_{1}%
,\ldots,\beta_{n}\right)  \in\mathbb{C}^{n}$, with $\operatorname{Re}%
(\beta_{i})>0$ for all $i$. The following formula holds:%
\[%
{\displaystyle\int\nolimits_{K^{n}}}
f_{\boldsymbol{\alpha}}\left(  \xi\right)  \widehat{g}\left(  \xi\right)
\left\vert d^{n}\xi\right\vert _{K}=%
{\displaystyle\int\nolimits_{K}}
{\textstyle\prod\nolimits_{i=1}^{n}}
\left\vert x_{i}\right\vert _{K}^{-\alpha_{i}}g\left(  x\right)  \left\vert
d^{n}x\right\vert _{K},
\]
for $g\in\mathcal{H}_{\infty}\left(  K^{n}\right)  $. In particular, the
distribution $\mathcal{F}_{x\rightarrow\xi}\left[  f_{\boldsymbol{\alpha}%
}\left(  x\right)  \right]  $ is a $\mathcal{H}_{\infty}^{\ast}$-valued
function, which admits an analytic continuation to $\mathbb{C}^{n}$.
\end{proposition}

\begin{remark}
The formula given in the second part of Proposition \ref{Prop3} shows that
$\mathcal{F}_{x\rightarrow\xi}\left[  f_{\boldsymbol{1}}\left(  x\right)
\right]  $, with $\boldsymbol{1}=(1,\ldots,1)\in\mathbb{C}^{n}$, is a
well-defined $\mathcal{H}_{\infty}^{\ast}$-valued function.
\end{remark}

\begin{proof}
(i) By the existence of $\boldsymbol{D}^{\beta\ast}$, we have
\[
\left[  \boldsymbol{D}^{\beta\ast}\mathcal{F}\left\{
{\textstyle\prod\nolimits_{i=1}^{n}}
\left\vert x_{i}\right\vert _{K}^{\overline{\alpha}_{i}-1}\right\}  ,g\right]
=\left[  \mathcal{F}\left\{
{\textstyle\prod\nolimits_{i=1}^{n}}
\left\vert x_{i}\right\vert _{K}^{\overline{\alpha}_{i}-1}\right\}
,\boldsymbol{D}^{\beta}g\right]  .
\]
Now%
\begin{gather*}
\left[  \mathcal{F}\left\{
{\textstyle\prod\nolimits_{i=1}^{n}}
\left\vert x_{i}\right\vert _{K}^{\overline{\alpha}_{i}-1}\right\}
,\boldsymbol{D}^{\boldsymbol{\beta}}g\right]  =%
{\displaystyle\int\nolimits_{K^{n}}}
{\textstyle\prod\nolimits_{i=1}^{n}}
\left\vert x_{i}\right\vert _{K}^{\alpha_{i}-1}\text{ }\widehat{\boldsymbol{D}%
^{\boldsymbol{\beta}}g}\left(  x\right)  \left\vert d^{n}x\right\vert _{K}\\
=%
{\displaystyle\int\nolimits_{K^{n}}}
{\textstyle\prod\nolimits_{i=1}^{n}}
\left\vert x_{i}\right\vert _{K}^{\alpha_{i}+\beta_{i}-1}\widehat{g}\left(
\xi\right)  \left\vert d^{n}\xi\right\vert _{K}=\left[  \mathcal{F}\left\{
{\textstyle\prod\nolimits_{i=1}^{n}}
\left\vert x_{i}\right\vert _{K}^{\overline{\alpha}_{i}+\overline{\beta}%
_{i}-1}\right\}  ,g\right]  .
\end{gather*}

Now, since $%
{\textstyle\prod\nolimits_{i=1}^{n}}
\left\vert x_{i}\right\vert _{K}^{\overline{\alpha}_{i}-1}$, and then
$\mathcal{F}\left\{
{\textstyle\prod\nolimits_{i=1}^{n}}
\left\vert x_{i}\right\vert _{K}^{\overline{\alpha}_{i}-1}\right\}  $, is a
$\mathcal{H}_{\infty}^{\ast}$-valued function, which is holomorphic in
$\boldsymbol{\alpha}\in\mathbb{C}^{n}$ in $\operatorname{Re}\left(  \alpha
_{i}\right)  >0$ for all \ $i$, and $\mathcal{F}\left\{
{\textstyle\prod\nolimits_{i=1}^{n}}
\left\vert x_{i}\right\vert _{K}^{\overline{\alpha}_{i}+\overline{\beta}%
_{i}-1}\right\}  $ is a $\mathcal{H}_{\infty}^{\ast}$-valued function, which
is holomorphic in $\boldsymbol{\alpha}\in\mathbb{C}^{n}$ in $\operatorname{Re}%
\left(  \alpha_{i}\right)  >-\operatorname{Re}\left(  \beta_{i}\right)  $,
$i=1,\ldots,n$, cf. Lemma \ref{lemma4C}, we conclude that $\mathcal{F}\left\{
%
{\textstyle\prod\nolimits_{i=1}^{n}}
\left\vert x_{i}\right\vert _{K}^{\overline{\alpha}_{i}-1}\right\}  $ has an
analytic extension to the half-plane $\operatorname{Re}\left(  \alpha
_{i}\right)  >-\operatorname{Re}\left(  \beta_{i}\right)  $, $i=1,\ldots,n$.
By using the fact that $\boldsymbol{\beta}$ is arbitrary, the principle of
analytic continuation assures the existence of an analytic continuation of
$\mathcal{F}\left\{
{\textstyle\prod\nolimits_{i=1}^{n}}
\left\vert x_{i}\right\vert _{K}^{\overline{\alpha}_{i}-1}\right\}  $ to the
whole $\mathbb{C}^{n}$.

(ii) By (\ref{Eq4}), we have
\begin{equation}%
{\displaystyle\int\nolimits_{K^{n}}}
f_{\boldsymbol{\alpha}}\left(  \xi\right)  \widehat{\phi}\left(  \xi\right)
\left\vert d^{n}\xi\right\vert _{K}=%
{\displaystyle\int\nolimits_{K^{n}}}
{\textstyle\prod\nolimits_{i=1}^{n}}
\left\vert x_{i}\right\vert _{K}^{-\alpha_{i}}\phi\left(  x\right)  \left\vert
d^{n}x\right\vert _{K}, \label{Eq5}%
\end{equation}
for $\phi\in D\left(  K^{n}\right)  $ and $\operatorname{Re}(\alpha_{i})\neq1$
for $i=1,\ldots,n$. By Lemma \ref{lemma4C}, the left-hand side of (\ref{Eq5})
defines a $\mathcal{H}_{\infty}^{\ast}$-valued function, which is holomorphic
in \ $\boldsymbol{\alpha}\in\mathbb{C}^{n}$ in the half-plane
\[
\operatorname{Re}(\alpha_{i})>0\text{ for }i=1,\ldots,n.
\]
By switching $\widehat{\phi}$ and $\phi$ in (\ref{Eq5}) and applying Lemma
\ref{lemma4C}, we obtain that the functional in the right-hand side of
(\ref{Eq5}) defines a $\mathcal{H}_{\infty}^{\ast}$-valued function, which is
holomorphic in \ $\boldsymbol{\alpha}\in\mathbb{C}^{n}$ in the half-plane%
\[
\operatorname{Re}(\alpha_{i})<1\text{ for }i=1,\ldots,n.
\]
Then both functionals appearing in (\ref{Eq5}) are holomorphic in the open
subset%
\begin{equation}
\operatorname{Re}(\alpha_{i})\in\left(  0,1\right)  \text{ for }i=1,\ldots,n.
\label{Eq6}%
\end{equation}
Now, by the principle of analytic continuation, \ and the fact that
$\mathcal{D}$ is dense in $\mathcal{H}_{\infty}$, cf. \cite[Lemma
3.4.]{Zuniga2016}, formula (\ref{Eq5}) is valid when $f_{\boldsymbol{\alpha}%
}\left(  \xi\right)  $\ and $%
{\textstyle\prod\nolimits_{i=1}^{n}}
\left\vert x_{i}\right\vert _{K}^{-\alpha_{i}}$ are considered $\mathcal{H}%
_{\infty}^{\ast}$-valued functions of $\boldsymbol{\alpha}$, with
$\operatorname{Re}(\alpha_{i})\neq1$ for all $i$. Notice that (\ref{Eq5}) can
be re-written as%
\[
\left[  \mathcal{F}_{x\rightarrow\xi}\left\{  f_{\overline{\boldsymbol{\alpha
}}}\left(  x\right)  \right\}  ,g\right]  =\left[  \mathcal{F}_{y\rightarrow
x}\left\{
{\textstyle\prod\nolimits_{i=1}^{n}}
\left\vert y_{i}\right\vert _{K}^{\overline{\alpha}_{i}}\right\}  ,\widehat
{g}\right]  ,\text{ }g\in\mathcal{H}_{\infty}^{\ast},
\]
where the Fourier transforms are understood in $\mathcal{D}^{\prime}$.
Finally, since $\mathcal{F}_{y\rightarrow x}\left\{
{\textstyle\prod\nolimits_{i=1}^{n}}
\left\vert y_{i}\right\vert _{K}^{\overline{\alpha}_{i}}\right\}  $ admits an
analytic continuation to $\mathbb{C}^{n}$, we conclude that $\mathcal{F}%
_{x\rightarrow\xi}\left\{  f_{\overline{\boldsymbol{\alpha}}}\left(  x\right)
\right\}  $ also admits an analytic continuation to $\mathbb{C}^{n}$.
\end{proof}

\begin{remark}
It is important to mention that the verification that the functionals
appearing in both sides of the formula given in Proposition \ref{Prop3}-(ii)
have a common domain of regularity is an essential matter. There are several
examples of functional equations for local zeta functions where the domain of
regularity is `the empty set'. For an in-depth discussion of this phenomenon
the reader may consult \cite[pp. 551-552]{Rubin} and the references therein.
\end{remark}

\subsection{\label{Sect_elelmetary_integrals}Meromorphic continuation of
elementary integrals in $\mathcal{H}_{\infty}$}

We fix $\boldsymbol{N}=\left(  N_{1},\ldots,N_{n}\right)  \in\left(
\mathbb{N\smallsetminus}\left\{  0\right\}  \right)  ^{n}$, $\boldsymbol{v}%
=\left(  v_{1},\ldots,v_{n}\right)  \in\left(  \mathbb{N\smallsetminus
}\left\{  0,1\right\}  \right)  ^{n}$. The \textit{elementary integral}
\textit{attached to} $\left(  \boldsymbol{N},\boldsymbol{v}\right)  $
\textit{and} $g\in\mathcal{H}_{\infty}\left(  K^{n}\right)  $ is defined as%
\[
E_{\widehat{g}}\left(  s;\boldsymbol{N},\boldsymbol{v}\right)  =\int
\nolimits_{K^{n}}\prod\nolimits_{i=1}^{r}\left\vert \xi_{i}\right\vert
_{K}^{N_{i}s+v_{i}-1}\widehat{g}\left(  \xi\right)  \left\vert d^{n}%
\xi\right\vert _{K},
\]
where $1\leq r\leq n$. By Lemma \ref{lemma4C}, $E_{\widehat{g}}\left(
s;\boldsymbol{N},\boldsymbol{v}\right)  $ defines $\mathcal{H}_{\infty}^{\ast
}$-valued holomorphic\ function of $s$ in the half-plane $\operatorname{Re}%
(s)>\max_{1\leq i\leq n}\frac{-v_{i}}{N_{i}}$.

\begin{definition}
\label{Def_sequence}Let $\left\{  \gamma_{i}\right\}  _{i\in
\mathbb{N\smallsetminus}\left\{  0\right\}  }$ be a sequence of positive real
numbers such that $\gamma_{1}\geq1$.\ The generalized \ arithmetic progression
generated by $\left\{  \gamma_{i}\right\}  _{i\in\mathbb{N}}$ is the sequence
$M=\left\{  m_{i}\right\}  _{i\in\mathbb{N}}$ of real numbers defined as: (1)
$m_{0}=0$ and $m_{1}=\gamma_{1}-1$; (2) $m_{l}=\sum_{j=1}^{l}\gamma_{j}$ for
$l\geq2$.
\end{definition}

\begin{proposition}
\label{Prop4}Let $M_{j}=\left\{  m_{i}^{(j)}\right\}  _{i\in\mathbb{N}}$, for
$j=1,\ldots,n$, be given generalized \ arithmetic progressions. Then
$E_{\widehat{g}}\left(  s;\boldsymbol{N},\boldsymbol{v}\right)  $ has an
analytic continuation the whole complex plane as a $\mathcal{H}_{\infty}%
^{\ast}\left(  K^{n}\right)  $-valued \ meromorphic function of $s$, denoted
again as $E_{\widehat{g}}\left(  s;\boldsymbol{N},\boldsymbol{v}\right)  $,
and the real parts of the possible poles of $E_{\widehat{g}}\left(
s;\boldsymbol{N},\boldsymbol{v}\right)  $ belong to the set $\cup_{i=1}%
^{r}\frac{-\left(  v_{i}+M_{i}\right)  }{N_{i}}$. In particular the real parts
of the possible poles are negative real numbers.
\end{proposition}

\begin{proof}
First, without loss of generality, we may assume that $r=n$. Indeed, by using
the fact that $\left\vert \widehat{g}\left(  \xi\right)  \right\vert
\prod\nolimits_{i=1}^{r}\left\vert \xi_{i}\right\vert _{K}^{N_{i}%
\operatorname{Re}\left(  s_{i}\right)  +v_{i}-1}\in L^{1}\left(
K^{n},\left\vert d^{n}\xi\right\vert _{K}\right)  $ for $\operatorname{Re}%
(s)>\max_{1\leq i\leq r}\frac{-v_{i}}{N_{i}}$, and applying Fubini's theorem,
we have%
\[
E_{\widehat{g}}\left(  s;\boldsymbol{N},\boldsymbol{v}\right)  =\int
\nolimits_{K^{n}}\prod\nolimits_{i=1}^{r}\left\vert \xi_{i}\right\vert
_{K}^{Ns_{i}+v_{i}-1}\mathcal{F}\left(  \boldsymbol{P}_{_{\left(  \xi
_{r+1},\ldots,\xi_{n}\right)  ,0}}g\right)  \left(  \xi_{1},\ldots,\xi
_{r}\right)  \left\vert d^{r}\xi\right\vert _{K},
\]
with $\boldsymbol{P}_{_{\left(  \xi_{r+1},\ldots,\xi_{n}\right)  ,0}}g\left(
\xi_{1},\ldots,\xi_{r}\right)  \in\mathcal{H}_{\infty}\left(  K^{r}\right)  $,
cf. Lemma \ref{lemma4B_1}. Consequently, we can take $r=n$. By taking
$\alpha_{i}=N_{i}s+v_{i}-1$, $i=1,\ldots,n$, and applying Proposition
\ref{Prop3}-(i), we have that $E_{\widehat{g}}\left(  s;\boldsymbol{N}%
,\boldsymbol{v}\right)  $ has an analytic continuation the whole complex plane
as a $\mathcal{H}_{\infty}^{\ast}\left(  K^{n}\right)  $-valued \ meromorphic
function of $s$.

We now proceed to describe the possible poles of the analytic continuation of
$E_{\widehat{g}}\left(  s;\boldsymbol{N},\boldsymbol{v}\right)  $: we start
with the integral
\[
E_{\widehat{g}}\left(  s;\boldsymbol{N},\boldsymbol{v}\right)  =\left[
\mathcal{F}\left[  \prod\nolimits_{i=1}^{n}\left\vert \xi_{i}\right\vert
_{K}^{N_{i}\overline{s}+v_{i}-1}\right]  ,g\right]
\]
which holomorphic in $\operatorname{Re}(s)>\max_{1\leq i\leq n}\frac{-v_{i}%
}{N_{i}}$. Now, we take $\boldsymbol{\gamma}_{1}=\left(  \gamma_{1}^{\left(
1\right)  },\ldots,\gamma_{1}^{\left(  n\right)  }\right)  $ an `arbitrary
vector' in $\left(  \mathbb{R}_{+}\mathbb{\smallsetminus}\left\{  0\right\}
\right)  ^{n}$, and consider the integral
\[
\left[  \boldsymbol{D}^{\boldsymbol{\gamma}_{1}\ast}\mathcal{F}\left[
\prod\nolimits_{i=1}^{n}\left\vert \xi_{i}\right\vert _{K}^{N_{i}\overline
{s}+v_{i}-1}\right]  ,g\right]  =\int\nolimits_{K^{n}}\prod\nolimits_{i=1}%
^{n}\left\vert \xi_{i}\right\vert _{K}^{Ns_{i}+v_{i}-1}\widehat{\boldsymbol{D}%
^{\boldsymbol{\gamma}}g}\left(  \xi\right)  \left\vert d^{n}\xi\right\vert
_{K}%
\]
which holomorphic in $\operatorname{Re}(s)>\max_{1\leq i\leq n}\frac{-v_{i}%
}{N_{i}}$ since $\boldsymbol{D}^{\boldsymbol{\gamma}_{1}\ast}:\mathcal{H}%
_{\infty}^{\ast}\rightarrow\mathcal{H}_{\infty}^{\ast}$. By using that
\begin{align*}
&  \int\nolimits_{K^{n}}\prod\nolimits_{i=1}^{n}\left\vert \xi_{i}\right\vert
_{K}^{Ns_{i}+v_{i}-1}\widehat{\boldsymbol{D}^{\boldsymbol{\gamma}_{1}}%
g}\left(  \xi\right)  \left\vert d^{n}\xi\right\vert _{K}\\
&  =\prod\nolimits_{i=1}^{n}\Gamma\left(  N_{i}s+v_{i}\right)  \int
\nolimits_{K^{n}}\prod\nolimits_{i=1}^{n}f_{N_{i}s+v_{i}}\left(  \xi\right)
\text{ }\widehat{\boldsymbol{D}^{\boldsymbol{\gamma}_{1}}g}\left(  \xi\right)
\left\vert d^{n}\xi\right\vert _{K}\\
&  =\prod\nolimits_{i=1}^{n}\Gamma\left(  N_{i}s+v_{i}+\gamma_{1}^{\left(
i\right)  }\right)  \int\nolimits_{K^{n}}\prod\nolimits_{i=1}^{n}%
f_{N_{i}s+v_{i}+\gamma_{1}^{\left(  i\right)  }}\left(  \xi\right)  \text{
}\widehat{g}\left(  \xi\right)  \left\vert d^{n}\xi\right\vert _{K},
\end{align*}
and using Proposition \ref{Prop3}-(ii),%
\begin{align}
&  \prod\nolimits_{i=1}^{n}\Gamma\left(  N_{i}s+v_{i}\right)  \int
\nolimits_{K^{n}}\prod\nolimits_{i=1}^{n}\left\vert \xi_{i}\right\vert
_{K}^{-N_{i}s-v_{i}}\text{ }\boldsymbol{D}^{\boldsymbol{\gamma}_{1}}g\left(
\xi\right)  \left\vert d^{n}\xi\right\vert _{K}\label{Eq24}\\
&  =\prod\nolimits_{i=1}^{n}\Gamma\left(  N_{i}s+v_{i}+\gamma_{1}^{\left(
i\right)  }\right)  \int\nolimits_{K^{n}}\prod\nolimits_{i=1}^{n}\left\vert
\xi_{i}\right\vert _{K}^{-N_{i}s-v_{i}-\gamma_{1}^{\left(  i\right)  }}\text{
}g\left(  \xi\right)  \left\vert d^{n}\xi\right\vert _{K}.\nonumber
\end{align}
We take $\gamma_{1}^{\left(  i\right)  }=\beta N_{i}$ for $i=1,\ldots,n$, with
$\beta\in\mathbb{N\smallsetminus}\left\{  0\right\}  $ arbitrary, and set
$z=-\left(  s+\beta\right)  $, thus $-N_{i}s-v_{i}-\gamma_{1}^{\left(
i\right)  }=N_{i}z-v_{i}$ and $-N_{i}s-v_{i}=N_{i}z-v_{i}+\gamma_{1}^{\left(
i\right)  }$, therefore formula (\ref{Eq24}) becomes%
\begin{align}
&  \int\nolimits_{K^{n}}\prod\nolimits_{i=1}^{n}\left\vert \xi_{i}\right\vert
_{K}^{N_{i}z-v_{i}}\text{ }g\left(  \xi\right)  \left\vert d^{n}\xi\right\vert
_{K}\label{Eq26}\\
&  =\frac{\prod\nolimits_{i=1}^{n}\Gamma\left(  -N_{i}z+v_{i}-\gamma
_{1}^{\left(  i\right)  }\right)  }{\prod\nolimits_{i=1}^{n}\Gamma\left(
-N_{i}z+v_{i}\right)  }\int\nolimits_{K^{n}}\prod\nolimits_{i=1}^{n}\left\vert
\xi_{i}\right\vert _{K}^{N_{i}z-v_{i}+\gamma_{1}^{\left(  i\right)  }}\text{
}\boldsymbol{D}^{\boldsymbol{\gamma}_{1}}g\left(  \xi\right)  \left\vert
d^{n}\xi\right\vert _{K}.\nonumber
\end{align}
Now the integral in the left-hand side of (\ref{Eq26}) is holomorphic on
$\operatorname{Re}(z)>\max_{i}\frac{-1+v_{i}}{N_{i}}$, and the integral in the
right-hand side of (\ref{Eq26}) is holomorphic in $\operatorname{Re}%
(z)>-\beta+\max_{i}\frac{-1+v_{i}}{N_{i}}$, and the factor%
\[
\frac{\prod\nolimits_{i=1}^{n}\Gamma\left(  -N_{i}z+v_{i}-\gamma_{1}^{\left(
i\right)  }\right)  }{\prod\nolimits_{i=1}^{n}\Gamma\left(  -N_{i}%
z+v_{i}\right)  }=\prod\nolimits_{i=1}^{n}\left(  \frac{1-q^{-N_{i}%
z+v_{i}-\gamma_{1}^{\left(  i\right)  }-1}}{1-q^{N_{i}z-v_{i}+\gamma
_{1}^{\left(  i\right)  }}}\frac{1-q^{N_{i}z-v_{i}}}{1-q^{-N_{i}z+v_{i}-1}%
}\right)
\]
gives poles with real parts $\operatorname{Re}(z)=\frac{v_{i}}{N_{i}}-\beta$
or $\operatorname{Re}(z)=\frac{v_{i}}{N_{i}}-\frac{1}{N_{i}}$. Therefore, in
terms of the variable $s$, \ the integral in the right-hand side of
(\ref{Eq24}), which is holomorphic in $\operatorname{Re}(s)<-\beta+\max
_{i}\frac{1-v_{i}}{N_{i}}$, has a meromorphic continuation in the half-plane
$\operatorname{Re}(s)<\max_{i}\frac{1-v_{i}}{N_{i}}$, with possible poles
having real parts in the set
\begin{equation}
\cup_{i=1}^{n}\frac{-v_{i}}{N_{i}}\cup\cup_{i=1}^{n}\frac{-\left(
v_{i}+\gamma_{1}^{\left(  i\right)  }-1\right)  }{N_{i}}. \label{poles_1}%
\end{equation}
Therefore, the real parts of the possible poles of%
\begin{align*}
\left[  \boldsymbol{D}^{\boldsymbol{\gamma}_{1}\ast}\mathcal{F}\left[
\prod\nolimits_{i=1}^{n}\left\vert \xi_{i}\right\vert _{K}^{N_{i}\overline
{s}+v_{i}-1}\right]  ,g\right]   &  =\int\nolimits_{K^{n}}\prod\nolimits_{i=1}%
^{n}\left\vert \xi_{i}\right\vert _{K}^{Ns_{i}+v_{i}+\gamma_{1}^{\left(
i\right)  }-1}\widehat{g}\left(  \xi\right)  \left\vert d^{n}\xi\right\vert
_{K}\\
&  =E_{\widehat{g}}\left(  s;\boldsymbol{N},\boldsymbol{v}+\boldsymbol{\gamma
}_{1}\right)
\end{align*}
belong to set (\ref{poles_1}). \ We repeat the calculation starting with
$E_{\widehat{g}}\left(  s;\boldsymbol{N},\boldsymbol{v}+\boldsymbol{\gamma
}_{1}\right)  $ and $\boldsymbol{\gamma}_{2}=\left(  \gamma_{2}^{\left(
1\right)  },\ldots,\gamma_{2}^{\left(  n\right)  }\right)  \in\left(
\mathbb{R}_{+}\mathbb{\smallsetminus}\left\{  0\right\}  \right)  ^{n}$, to
obtain that the real parts of the possible poles of $E_{\widehat{g}}\left(
s;\boldsymbol{N},\boldsymbol{v}+\boldsymbol{\gamma}_{1}+\boldsymbol{\gamma
}_{2}\right)  $ belong to the set
\begin{equation}
\cup_{i=1}^{n}\frac{-v_{i}}{N_{i}}\cup\cup_{i=1}^{n}\frac{-\left(
v_{i}+\gamma_{1}^{\left(  i\right)  }-1\right)  }{N_{i}}\cup\cup_{i=1}%
^{n}\frac{-\left(  v_{i}+\gamma_{1}^{\left(  i\right)  }+\gamma_{2}^{\left(
i\right)  }-1\right)  }{N_{i}}. \label{poles_2}%
\end{equation}
By proceeding inductively, we obtain the description of the real parts of the
possible poles of $E_{\widehat{g}}\left(  s;\boldsymbol{N},\boldsymbol{v}%
\right)  $\ announced. Finally, by Lemma \ref{lemma4C}, $E_{\widehat{g}%
}\left(  s;\boldsymbol{N},\boldsymbol{v}\right)  $\ is holomorphic in the
half-plane $\operatorname{Re}(s)>-1$ and consequently all the real parts of
the poles of $E_{\widehat{g}}\left(  s;\boldsymbol{N},\boldsymbol{v}\right)  $
are negative real numbers, and thus in (\ref{poles_2}) we have to take
$\gamma_{1}^{\left(  i\right)  }\geq1$.
\end{proof}

\begin{remark}
We notice that there is no canonical way of picking the sequence of vectors
$\left\{  \boldsymbol{\gamma}_{m}\right\}  _{m\in\mathbb{N}}$, this is the
reason for Definition \ref{Def_sequence}. On the other hand, integral
$\int_{K}\left\vert x\right\vert _{K}^{Ns+v-1}e^{-\left\vert x\right\vert
_{K}^{\alpha}}\left\vert dx\right\vert _{K}$, with $\alpha>0$, has a
meromorphic continuation to the whole complex plane with possible poles
belonging to $-\left(  \frac{v+\alpha\mathbb{N}}{N}\right)  $.
\end{remark}

\subsection{Meromorphic Continuation of Local Zeta Functions in $\mathcal{H}%
_{\infty}$}

Only in this section, $K$ denotes a non-Archimedean local field of
characteristic zero, i.e. $K$ is a finite extension of $\mathbb{Q}_{p}$, the
field of $p$-adic numbers. In this section, we show the existence of a
meromorphic continuation for the functional $\widehat{\left\vert
\widetilde{\mathfrak{f}}\right\vert _{K}^{\overline{s}}}$, with
$\operatorname{Re}(s)>0$. A key ingredient is Hironaka's desingularization
theorem (analytic version). For an in-depth discussion of this result as well
as, an introduction to the necessary material, the reader may consult
\cite{H}, \cite{Igusa} and the references therein.

\begin{notation}
We set $\mathfrak{f}^{-1}\left(  0\right)  _{\text{sing}}:=\left\{  \xi\in
K^{n};\mathfrak{f}\left(  \xi\right)  =\nabla\mathfrak{f}\left(  \xi\right)
=0\right\}  $. If $\Pi$ is a $K$-analytic map, then the pull-back of the
differential form $\bigwedge\limits_{1\leq i\leq n}d\xi_{i}$ by $\Pi$ is
classically denoted as $\Pi^{\ast}\left(  \bigwedge\limits_{1\leq i\leq n}%
d\xi_{i}\right)  $, but since we use `$\ast$' in connection with dual space of
$\mathcal{H}_{\infty}$ and the adjoint of an operator, we modify this
classical notation as $\Pi^{\underline{\ast}}\left(  \bigwedge\limits_{1\leq
i\leq n}d\xi_{i}\right)  $. Similarly, in the case of function, we use the
notation $\phi^{\underline{\ast}}=\phi\circ\Pi$. This notation is used only in
this section. We denote by ${\LARGE 1}_{A}\left(  x\right)  $ the
characteristic function of $A$.
\end{notation}

\begin{theorem}
[Hironaka \cite{H}]\label{Theorem_hironaka}Let $K$ denote a non-Archimedean
local field of characteristic zero and $\mathfrak{f}\left(  \xi\right)  $ a
non-constant polynomial in $K\left[  \xi_{1},\ldots,\xi_{n}\right]  $. Put
$X=K^{n}$. Then there exist an $n$-dimensional $K$-analytic manifold $Y$, a
finite set $\mathcal{E}=\left\{  E\right\}  $ of closed submanifolds of $Y$ of
codimension $1$ with a pair of positive integers $\left(  N_{E},v_{E}\right)
$ assigned to each $E$, and a proper $K$-analytic mapping $\Pi:Y\rightarrow X$
satisfying the following conditions:

\noindent(i) $\Pi$ is the composite map of a finite number of monomial
transformations each with smooth center;

\noindent(ii) $\left(  \mathfrak{f}\circ\Pi\right)  ^{-1}\left(  0\right)
=\cup_{E\in\mathcal{E}}E$ and $\Pi$ induces a $K$-bianalytic map%
\begin{equation}
Y\smallsetminus\Pi^{-1}\left(  \mathfrak{f}^{-1}\left(  0\right)
_{\text{sing}}\right)  \rightarrow X\smallsetminus\mathfrak{f}^{-1}\left(
0\right)  _{\text{sing}}; \label{change_of_variables_2}%
\end{equation}

\noindent(iii) at every point $b$ of $Y$ if $E_{1}$,\ldots,$E_{r}$ are all the
$E$ in $\mathcal{E}$ containing $b$ with respective local equations $y_{1}%
$,\ldots,$y_{r}$ around $b$ and $\left(  N_{E_{i}},v_{E_{i}}\right)  =\left(
N_{i},v_{i}\right)  $, then there exist local coordinates of $Y$ around $b$ of
the form $\left(  y_{1},\ldots,y_{r},y_{r+1},\ldots,y_{n}\right)  $ such that%
\begin{equation}
\left(  \mathfrak{f}\circ\Pi\right)  \left(  y\right)  =\varepsilon\left(
y\right)  \prod\limits_{1\leq i\leq r}y_{i}^{N_{i}}\text{, \ \ }%
\Pi^{\underline{\ast}}\left(  \bigwedge\limits_{1\leq i\leq n}d\xi_{i}\right)
=\eta\left(  y\right)  \prod\limits_{1\leq i\leq r}y_{i}^{v_{i}-1}\left(
\bigwedge\limits_{1\leq i\leq n}dy_{i}\right)  \label{local_expressions}%
\end{equation}
on some neighborhood of $b$, in which $\varepsilon\left(  y\right)  $,
$\eta\left(  y\right)  $ are units of the local ring $\mathcal{O}_{b}$ of $Y$
at $b$.
\end{theorem}

We call the pair $(Y,\Pi)$ \textit{an embedded resolution of singularities of
the map} $\mathfrak{f}:K^{n}\rightarrow K$. The set $\left\{  \left(
N_{E},v_{E}\right)  \right\}  _{E\in\mathcal{E}}$ is called \textit{the
numerical data of the resolution} $\left(  Y,\Pi\right)  $.

\begin{remark}
The following facts will be used later on: (i) $Y$ is a $2$-countable and
totally disconnected space. This follows from the fact that $\mathbb{Q}%
_{p}^{n}$ is a $2$-countable space, and from the fact that $Y$ is obtained by
a gluing a finite number of subspaces of $\mathbb{Q}_{p}^{n}$.

\noindent(ii) There exists a covering of $Y$ of the form $\cup_{m\in
\mathbb{N}}$ $U_{m}$ with each $U_{m}$ open and closed, $U_{i}\cap
U_{j}=\emptyset$ if $i\neq j$ and \ such that in local coordinates each
$U_{m}$ has the form $c_{m}+\left(  \pi^{e_{m}}R_{K}\right)  ^{n}$ with
$c_{m}\in K^{n}$ and $e_{m}\in\mathbb{N}$ for each $m$.

\noindent(iii) We denote by $\mathcal{D}(Y)$ the space of complex-valued
locally constant functions with compact support defined in $Y$, any such
function is a linear combination of characteristic functions of open and
compact subsets of $Y$, see e.g. \cite[Chapter 7]{Igusa}.

\noindent(iv) The differential form $%
{\textstyle\bigwedge\nolimits_{1\leq i\leq n}}
d\xi_{i}$ in $K^{n}$ (considered as $K$-analytic manifold) induces a measure,
denoted as $\left\vert
{\textstyle\bigwedge\nolimits_{1\leq i\leq n}}
d\xi_{i}\right\vert _{K}$, which agrees with normalized Haar measure of
$K^{n}$, see e.g. \cite[Chapter 7]{Igusa}.
\end{remark}

\begin{theorem}
\label{Theorem4} Assume that $K$ is a non-Archimedean local field of
characteristic zero and let $\mathfrak{f}$ denote an arbitrary element of
$R_{K}\left[  \xi_{1},\ldots,\xi_{n}\right]  \smallsetminus R_{K}$; take
$g\in\mathcal{H}_{\infty}\left(  K^{n}\right)  $, $s\in\mathbb{C}$, with
$\operatorname{Re}(s)>0$. Then $\left[  \widehat{\left\vert \widetilde
{\mathfrak{f}}\right\vert _{K}^{\overline{s}}},g\right]  $ defines a
$\mathcal{H}_{\infty}^{\ast}\left(  K^{n}\right)  $-valued holomorphic
function of $s$, which admits a meromorphic continuation, denoted again as
$\left[  \widehat{\left\vert \widetilde{\mathfrak{f}}\right\vert
_{K}^{\overline{s}}},g\right]  $, to the whole complex plane. Furthermore, if
$\Pi:Y\rightarrow X$, with $X=K^{n}$, $\mathcal{E}=\left\{  E\right\}  $ and
$\left(  N_{E},v_{E}\right)  $ as in Theorem \ref{Theorem_hironaka}, then the
possible real parts of the poles of $\left[  \widehat{\left\vert
\widetilde{\mathfrak{f}}\right\vert _{K}^{\overline{s}}},g\right]  $ are
negative real numbers belonging to the set
\[
\bigcup\limits_{E\in\mathcal{E}}\frac{-\left(  v_{E}+M_{E}\right)  }{N_{E}},
\]
where each $M_{E}$ is a generalized arithmetic progression.
\end{theorem}

\begin{proof}
The fact that $\left[  \widehat{\left\vert \widetilde{\mathfrak{f}}\right\vert
_{K}^{\overline{s}}},g\right]  $ defines a $\mathcal{H}_{\infty}^{\ast}\left(
K^{n}\right)  $-valued holomorphic function of $s$ in the half-plane
$\operatorname{Re}(s)>0$ follows from Lemma \ref{lemma4B}. To establish the
meromorphic continuation, we pick an embedded resolution of singularities of
the map $\mathfrak{f}:K^{n}\rightarrow K$ as in Theorem \ref{Theorem_hironaka}
and we use all the notation that was introduced there. We take $\phi
\in\mathcal{D}(K^{n})$ and use (\ref{change_of_variables_2}) as an analytic
change of variables in $\left[  \widehat{\left\vert \widetilde{\mathfrak{f}%
}\right\vert _{K}^{\overline{s}}},\phi\right]  $ to obtain%
\[
\left[  \widehat{\left\vert \widetilde{\mathfrak{f}}\right\vert _{K}%
^{\overline{s}}},\phi\right]  =\int\limits_{X\smallsetminus\mathfrak{f}%
^{-1}\left(  0\right)  }\left\vert \mathfrak{f}\circ\Pi\left(  y\right)
\right\vert _{K}^{\overline{s}}\left(  \phi\circ\Pi\right)  \left(  y\right)
\left\vert \Pi^{\underline{\ast}}\left(  \bigwedge\limits_{1\leq i\leq n}%
d\xi_{i}\right)  \right\vert _{K}\text{ for }\operatorname{Re}(\overline
{s})>0\text{.}%
\]
At very point $b\in Y$ we can take a chart $(V,h_{V})$ with $V$ open and
compact and coordinates $\left(  y_{1},\ldots,y_{r},y_{r+1},\ldots
,y_{n}\right)  $ such that formulas (\ref{local_expressions}) hold in $V$.
Since $\phi\circ\Pi$, $\left\vert \varepsilon\left(  y\right)  \right\vert
_{K}$, $\left\vert \eta\left(  y\right)  \right\vert _{K}$ are locally
constant functions, by subdividing $V$ as a finite union of disjoint open and
compact subsets $U_{m}$, we have $\phi\circ\Pi\mid_{U_{m}}=\phi\left(
\Pi\left(  b\right)  \right)  $, $\left\vert \varepsilon\left(  y\right)
\right\vert _{K}\mid_{U_{m}}=\left\vert \varepsilon\left(  b\right)
\right\vert _{K}$, $\left\vert \eta\left(  y\right)  \right\vert _{K}%
\mid_{U_{m}}=\left\vert \eta\left(  b\right)  \right\vert _{K}$ and further
$h_{V}\left(  U_{m}\right)  =\widetilde{y}_{m}+\pi^{e_{m}}R_{K}^{n}=B_{-e_{m}%
}^{n}\left(  \widetilde{y}_{m}\right)  $ for some $\widetilde{y}_{m}\in K^{n}%
$, $e_{m}\in\mathbb{N}$. Now, by using the fact that $Y$ is $2$-countable we
may assume that $\left\{  U_{m}\right\}  _{m\in\mathbb{N}}$ is a covering of
$Y$ consisting of open and compact subsets which are pairwise disjoint.
Consequently, we have, first, a map%
\[%
\begin{array}
[c]{ccc}%
\underline{\ast}:\mathcal{D}(X) & \rightarrow & \mathcal{D}(Y)\\
&  & \\
\phi & \rightarrow & \phi^{\underline{\ast}}:=\phi\circ\Pi,
\end{array}
\]
in addition, $\phi^{\underline{\ast}}\mid_{U_{m}}$ is an element of
$\mathcal{D}\left(  B_{-e_{m}}^{n}\left(  \widetilde{y}_{m}\right)  \right)
\hookrightarrow\mathcal{D}(K^{n})$, where $K^{n}$ is \ an affine space with
coordinates $\left(  y_{1},\ldots,y_{r},y_{r+1},\ldots,y_{n}\right)  $; and
second,%
\begin{equation}
\left[  \widehat{\left\vert \widetilde{\mathfrak{f}}\left(  \xi\right)
\right\vert _{K}^{\overline{s}}},\phi\left(  \xi\right)  \right]  =%
{\textstyle\sum\limits_{m\in\mathbb{N}}}
d_{m}\left[  \mathcal{F}\left(  \widetilde{{\LARGE 1}}_{B_{-e_{m}}^{n}\left(
\widetilde{y}_{m}\right)  }\left(  y\right)  \prod\nolimits_{i=1}%
^{r}\left\vert y_{i}\right\vert _{K}^{N_{i}\overline{s}+v_{i}-1}\right)
,\phi^{\underline{\ast}}\left(  y\right)  \right]  \text{ } \label{Eq30}%
\end{equation}
for $\operatorname{Re}(\overline{s})>0$, where ${\LARGE 1}_{B_{-e_{m}}%
^{n}\left(  \widetilde{y}_{m}\right)  }\left(  y\right)  $ is the
characteristic function of $\widetilde{y}_{m}+\pi^{e_{m}}R_{K}^{n}$, $1\leq
r=r(n)\leq n$, and ${\LARGE 1}_{B_{-e_{m}}^{n}\left(  \widetilde{y}%
_{m}\right)  }\left(  y\right)  \phi^{\underline{\ast}}\left(  y\right)  $ is
an element of $\mathcal{D}(K^{n})$. Since $\left(  \mathcal{D}(K^{n}%
),d\right)  $ is dense in $\left(  \mathcal{H}_{\infty},d\right)  $ and
\[
\text{ }\mathcal{F}\left(  \widetilde{{\LARGE 1}}_{B_{-e_{m}}^{n}\left(
\widetilde{y}_{m}\right)  }\left(  y\right)  \prod\nolimits_{i=1}%
^{r}\left\vert y_{i}\right\vert _{K}^{N\overline{s}+v_{i}-1}\right)
\in\mathcal{H}_{\infty}^{\ast}\left(  K^{n}\right)  \text{ for }%
\operatorname{Re}(\overline{s})>0\text{, }m\in\mathbb{N}\text{,}%
\]
then (\ref{Eq30}) extends to an equality between functionals in $\mathcal{H}%
_{\infty}^{\ast}$ in the half-plane $\operatorname{Re}(\overline{s})>0$. Then,
from Proposition \ref{Prop4} follows that $\widehat{\left\vert \widetilde
{\mathfrak{f}}\left(  \xi\right)  \right\vert _{K}^{s}}$ has a meromorphic
continuation to the whole complex plane as a $\mathcal{H}_{\infty}^{\ast}%
$-valued function and that the real parts of the possible poles belong to the
set%
\[
\bigcup\limits_{E\in\mathcal{E}\text{ }}\left\{  \frac{-\left(  v_{E}%
+M_{E}\right)  }{N_{E}}\right\}  .
\]

\end{proof}

\begin{remark}
In \cite[Chapter III, Section\ 5]{Igusa1} Igusa computed an embedded
resolution of singularities for a strongly non-degenerate form, with numerical
data $\left\{  \left(  1,1\right)  ,\left(  d,n\right)  \right\}  $. Then, the
Claim in Section \ref{Sect_Local_zeta_forms} agrees with Theorem
\ref{Theorem4}.
\end{remark}

\section{Fundamental solutions and local zeta functions}

\begin{theorem}
\label{Theorem4A}Let $\mathfrak{f}$ be a non-constant polynomial with
coefficients in $R_{K}$, with $K$ a non-Archimedean local field of arbitrary
characteristic. Then, the following assertions are equivalent:

\noindent(i) there exists $E\in\mathcal{H}_{\infty}^{\ast}$ such that
$\widehat{E}\left\vert \mathfrak{f}\right\vert _{K}=1$ in $L^{2}$;

\noindent(ii) set $\boldsymbol{A}(\partial,\mathfrak{f})g=\mathcal{F}%
^{-1}\left(  \left\vert \mathfrak{f}\right\vert _{K}\mathcal{F}\left(
g\right)  \right)  $ for $g\in Dom\left(  \boldsymbol{A}(\partial
,\mathfrak{f})\right)  :=\left\{  g\in L^{2};\left\vert \mathfrak{f}%
\right\vert _{K}\widehat{g}\in L^{2}\right\}  $. There exists $E\in
\mathcal{H}_{\infty}^{\ast}$ such that $\boldsymbol{A}^{\ast}(\partial
,\mathfrak{f})E=\delta$ in $\mathcal{H}_{\infty}^{\ast}$;

\noindent(iii) there exists $E\in\mathcal{H}_{\infty}^{\ast}$ such that $E\ast
h\in\mathcal{H}_{\infty}^{\ast}$ for any $h\in\mathcal{H}_{\infty}$, and
$u=E\ast g$ is a solution of $\boldsymbol{A}^{\ast}(\partial,\mathfrak{f})u=g$
in $\mathcal{H}_{\infty}^{\ast}$, for any $g\in\mathcal{H}_{\infty}$.
\end{theorem}

\begin{definition}
The functional $E\in\mathcal{H}_{\infty}^{\ast}$ is called a fundamental
solution for $\boldsymbol{A}^{\ast}(\partial,\mathfrak{f})$.
\end{definition}

\begin{proof}
(i)$\Rightarrow$(ii) Since $\boldsymbol{A}(\partial,\mathfrak{f}%
):\mathcal{H}_{\infty}\rightarrow\mathcal{H}_{\infty}$ is a continuous
operator, cf. Lemma \ref{lemma3A}, and $E\in\mathcal{H}_{\infty}^{\ast}$,
\begin{align*}
\left[  \boldsymbol{A}^{\ast}(\partial,\mathfrak{f})E,g\right]   &  =\left[
E,\boldsymbol{A}(\partial,\mathfrak{f})g\right]  =%
{\textstyle\int\nolimits_{K^{n}}}
\overline{\widehat{E}}\left\vert \mathfrak{f}\right\vert _{K}\widehat
{g}\left\vert d^{n}\xi\right\vert _{K}\\
&  =%
{\textstyle\int\nolimits_{K^{n}}}
\overline{\widehat{E}\left\vert \mathfrak{f}\right\vert _{K}}\widehat
{g}\left\vert d^{n}\xi\right\vert _{K}=%
{\textstyle\int\nolimits_{K^{n}}}
\widehat{g}\left\vert d^{n}\xi\right\vert _{K}=\left[  \delta,g\right]  ,
\end{align*}
for $g\in\mathcal{H}_{\infty}$.

(ii)$\Rightarrow$(i) $\left[  \boldsymbol{A}^{\ast}(\partial,\mathfrak{f}%
)E,g\right]  =\left[  \delta,g\right]  $ implies%
\[%
{\textstyle\int\nolimits_{K^{n}}}
\overline{\widehat{E}}\left\vert \mathfrak{f}\right\vert _{K}\widehat
{g}\left\vert d^{n}\xi\right\vert _{K}=%
{\textstyle\int\nolimits_{K^{n}}}
\widehat{g}\left\vert d^{n}\xi\right\vert _{K},
\]
i.e.
\[%
{\textstyle\int\nolimits_{K^{n}}}
\left\{  \overline{\widehat{E}\left\vert \mathfrak{f}\right\vert _{K}%
-1}\right\}  \widehat{g}\left\vert d^{n}\xi\right\vert _{K}=0
\]
for any $g\in\mathcal{H}_{\infty}$. Since $\mathcal{H}_{\infty}$ is dense in
$L^{2}$, because $\mathcal{D}\hookrightarrow\mathcal{H}_{\infty}$, the
functional $g\rightarrow\int_{K^{n}}\left\{  \overline{\widehat{E}\left\vert
\mathfrak{f}\right\vert _{K}-1}\right\}  \widehat{g}\left\vert d^{n}%
\xi\right\vert _{K}$ extends to $L^{2}$ as the zero functional, i.e. the
function $\overline{\widehat{E}\left\vert \mathfrak{f}\right\vert _{K}-1}$ is
orthogonal to any $\widehat{g}\in L^{2}$, which implies that $\widehat
{E}\left\vert \mathfrak{f}\right\vert _{K}=1$ in $L^{2}$.

(iii) $\Rightarrow$(i) Take $h\in\mathcal{H}_{\infty}$, and $u=E\ast
h\in\mathcal{H}_{\infty}^{\ast}$, then%
\[
\left[  \boldsymbol{A}^{\ast}(\partial,\mathfrak{f})u,h\right]  =\left[
u,\boldsymbol{A}(\partial,\mathfrak{f})h\right]  =\left[  g,h\right]
\]
i.e.%
\begin{equation}%
{\textstyle\int\nolimits_{K^{n}}}
\overline{\widehat{E\ast g}}\left\vert \mathfrak{f}\right\vert _{K}\widehat
{g}\left\vert d^{n}\xi\right\vert _{K}=%
{\textstyle\int\nolimits_{K^{n}}}
\overline{\widehat{g}}\widehat{h}\left\vert d^{n}\xi\right\vert _{K}.
\label{Eq7}%
\end{equation}
By using that $E\ast h\in\mathcal{D}^{\prime}\left(  K^{n}\right)  $, see
(\ref{H_infinity_*}), we have $\widehat{E\ast h}=\widehat{E}\widehat{h}$ in
$\mathcal{D}^{\prime}\left(  K^{n}\right)  $, and thus (\ref{Eq7}) becomes
\[%
{\textstyle\int\nolimits_{K^{n}}}
\overline{\left\{  \left(  \widehat{E}\left\vert \mathfrak{f}\right\vert
_{K}-1\right)  \widehat{g}\right\}  }\widehat{h}\left\vert d^{n}\xi\right\vert
_{K}=0,
\]
which implies that $\left(  \widehat{E}\left\vert \mathfrak{f}\right\vert
_{K}-1\right)  \widehat{g}=0$ in $L^{2}$ for any $g\in\mathcal{H}_{\infty}$,
and hence $\widehat{E}\left\vert \mathfrak{f}\right\vert _{K}=1$ in $L^{2}$.

(ii)$\Rightarrow$(iii) First $E\ast g$ exists in $\mathcal{D}^{\prime}\left(
K^{n}\right)  $ if and only if $\widehat{E}\widehat{g}\in\mathcal{D}^{\prime
}\left(  K^{n}\right)  $, see e.g. \cite[p.115]{V-V-Z}. We check this last
condition: taking $\theta\in\mathcal{D}\left(  K^{n}\right)  $, we have for
some non-negative integer $l$ that%
\begin{align*}
\left\vert
{\textstyle\int\nolimits_{K^{n}}}
\widehat{E}\widehat{g}\theta\left\vert d^{n}\xi\right\vert _{K}\right\vert  &
=\left\vert
{\textstyle\int\nolimits_{K^{n}}}
\widehat{E}\mathcal{F}\left(  g\ast\mathcal{F}^{-1}\left(  \theta\right)
\right)  \left\vert d^{n}\xi\right\vert _{K}\right\vert \leq\left\Vert
E\right\Vert _{-l}\left\Vert g\ast\mathcal{F}^{-1}\theta\right\Vert _{l}\\
&  \leq\left\Vert \theta\right\Vert _{L^{\infty}}\left\Vert E\right\Vert
_{-l}\left\Vert g\right\Vert _{l}.
\end{align*}
This shows that $E\ast g$ $\in$ $\mathcal{D}^{\prime}\left(  K^{n}\right)  $
and that $\widehat{E\ast g}=$ $\widehat{E}\widehat{g}$ in $\mathcal{D}%
^{\prime}\left(  K^{n}\right)  $. On the other hand, $E\ast g$ $\in$
$\mathcal{H}_{\infty}^{\ast}\left(  K^{n}\right)  $ for any $g\in
\mathcal{H}_{\infty}\left(  K^{n}\right)  $, because%
\[
\left\Vert E\ast g\right\Vert _{-l}^{2}=%
{\textstyle\int\nolimits_{K^{n}}}
\left[  \xi\right]  _{K}^{-l}\left\vert \widehat{E}\right\vert ^{2}\left\vert
\widehat{g}\right\vert ^{2}\left\vert d^{n}\xi\right\vert _{K}\leq\left\Vert
\widehat{g}\right\Vert _{L^{\infty}}^{2}\left\Vert E\right\Vert _{-l}^{2}%
\]
for any positive integer $l$, since $\widehat{g}\in C_{0}\left(  K^{n}\right)
$, cf. Lemma \ref{Lemma2A}-(v). Now by using that (i)$\Leftrightarrow$(ii), we
have%
\begin{align*}
\left[  \boldsymbol{A}^{\ast}(\partial,\mathfrak{f})u,h\right]   &  =\left[
E\ast g,\boldsymbol{A}(\partial,\mathfrak{f})h\right]  =%
{\textstyle\int\nolimits_{K^{n}}}
\overline{\widehat{E}\widehat{g}}\left\vert \mathfrak{f}\right\vert
_{K}\widehat{h}\left\vert d^{n}\xi\right\vert _{K}\\
&  =%
{\textstyle\int\nolimits_{K^{n}}}
\overline{\left\{  \widehat{E}\left\vert \mathfrak{f}\right\vert _{K}\right\}
}\overline{\widehat{g}}\widehat{h}\left\vert d^{n}\xi\right\vert _{K}=%
{\textstyle\int\nolimits_{K^{n}}}
\overline{\widehat{g}}\widehat{h}\left\vert d^{n}\xi\right\vert _{K}=\left[
g,h\right]  .
\end{align*}

\end{proof}

\begin{theorem}
\label{Theorem5}Let $\mathfrak{f}$ be a non-constant polynomial with
coefficients in $R_{K}$, with $K$ a non-Archimedean local field of \ arbitrary
characteristic. Assume that $\left[  \widehat{\left\vert \widetilde
{\mathfrak{f}}\right\vert _{K}^{\overline{s}}},g\right]  $ has a meromorphic
continuation to the whole complex plane as a $\mathcal{H}_{\infty}^{\ast
}\left(  K^{n}\right)  $-valued function of $s$, with poles having negative
real parts. Then there exists a fundamental solution for operator
$\boldsymbol{A}^{\ast}(\partial,\mathfrak{f})$.
\end{theorem}

\begin{proof}
The proof is based in the Gel'fand-Shilov method of analytic continuation, see
\cite[p. \ 65-67]{Igusa}. By the hypothesis that $\left[  \widehat{\left\vert
\widetilde{\mathfrak{f}}\right\vert _{K}^{\overline{s}}},g\right]  $ has an
analytic continuation to the whole complex plane, there exists a Laurent
expansion \ around $s=-1$ of the form
\[
\left[  \widehat{\left\vert \widetilde{\mathfrak{f}}\right\vert _{K}%
^{\overline{s}}},g\right]  =\int\limits_{K^{n}\smallsetminus\mathfrak{f}%
^{-1}\left(  0\right)  }\left\vert \mathfrak{f}\right\vert _{K}^{s}\widehat
{g}\text{ }\left\vert d^{n}\xi\right\vert _{K}=\sum\limits_{k\in\mathbb{Z}%
}\left[  T_{k},g\right]  \left(  s+1\right)  ^{k}%
\]
where $T_{k}\in$ $\mathcal{H}_{\infty}^{\ast}$ for $k\in\mathbb{Z}$. This fact
is established by using the ideas presented in \cite[p. \ 65-67]{Igusa}. Now,
\begin{align*}
\left[  \widehat{\left\vert \widetilde{\mathfrak{f}}\right\vert _{K}%
^{\overline{s}}},\boldsymbol{A}(\partial,\mathfrak{f})g\right]   &
=\int\limits_{K^{n}\smallsetminus\mathfrak{f}^{-1}\left(  0\right)
}\left\vert \mathfrak{f}\right\vert _{K}^{s+1}\widehat{g}\text{ }\left\vert
d^{n}\xi\right\vert _{K}=\sum\limits_{k\in\mathbb{Z}}\left[  T_{k}%
,\boldsymbol{A}(\partial,\mathfrak{f})g\right]  \left(  s+1\right)  ^{k}\\
&  =\left[  T_{0},\boldsymbol{A}(\partial,\mathfrak{f})g\right]
+\sum\limits_{k=1}^{\infty}\left[  T_{k},\boldsymbol{A}(\partial
,\mathfrak{f})g\right]  \left(  s+1\right)  ^{k},
\end{align*}
since $\int_{K^{n}\smallsetminus\mathfrak{f}^{-1}\left(  0\right)  }\left\vert
\mathfrak{f}\right\vert _{K}^{s+1}\widehat{g}$ $\left\vert d^{n}\xi\right\vert
_{K}$ does not have poles with real part $-1$. Therefore%
\begin{align*}
\lim_{s\rightarrow-1}\left[  \widehat{\left\vert \widetilde{\mathfrak{f}%
}\right\vert _{K}^{\overline{s}}},\boldsymbol{A}(\partial,\mathfrak{f}%
)g\right]   &  =\int\limits_{K^{n}\smallsetminus\mathfrak{f}^{-1}\left(
0\right)  }\left\vert \mathfrak{f}\right\vert _{K}^{s+1}\widehat{g}\text{
}\left\vert d^{n}\xi\right\vert _{K}=\left[  T_{0},\boldsymbol{A}%
(\partial,\mathfrak{f})g\right] \\
&  =\int\limits_{K^{n}\smallsetminus\mathfrak{f}^{-1}\left(  0\right)
}\widehat{g}\text{ }\left\vert d^{n}\xi\right\vert _{K},
\end{align*}
since $\widehat{g}\in L^{1}$, i.e. $\left[  T_{0},\boldsymbol{A}%
(\partial,\mathfrak{f})g\right]  =\left[  \delta,g\right]  $, for any
$g\in\mathcal{H}_{\infty}$, which implies that $\boldsymbol{A}^{\ast}%
(\partial,\mathfrak{f})T_{0}=\delta$ with $T_{0}\in\mathcal{H}_{\infty}^{\ast
}$.
\end{proof}


\begin{thebibliography}{99}                                                                                               %


\bibitem {A-K-S}Albeverio S., Khrennikov A. Yu., Shelkovich V. M., Theory of
$p$-adic distributions: linear and nonlinear models. Cambridge University
Press, 2010.

\bibitem {AVG}Arnold V. I., Gussein-Zade S. M. and Varchenko A. N.,
Singularit\'{e}s des applications diff\'{e}rentiables, Vol II. \'{E}ditions
Mir, Moscou, 1986.

\bibitem {Atiyah}Atiyah M. F., Resolution of singularities and division of
distributions, Comm. Pure Appl. Math. 23 (1970), 145--150.

\bibitem {Be-Bro}Belkale P. and Brosnan P., Periods and Igusa local zeta
functions, Int. Math. Res. Not. 49 (2003), 2655--2670.

\bibitem {Berthelot}Berthelot Pierre, Introduction \`{a} la th\'{e}orie
arithm\'{e}tique des $\mathcal{D}$-modules. Cohomologies $p$-adiques et
applications arithm\'{e}tiques, II, Ast\'{e}risque No. 279 (2002), 1--80.

\bibitem {Ber}Bernstein I. N., Modules over the ring of differential
operators; the study of fundamental solutions of equations with constant
coefficients, Functional Analysis and its Applications 5, No.2, 1-16 (1972).

\bibitem {B-G-Gonzalez-Dom}Bollini C.G., Giambiagi J. J., Gonz\'{a}lez
Dom\'{\i}nguez A., Analytic regularization and the divergencies of quantum
field theories, Il Nuovo Cimiento XXXI, no. 3 (1964) 550-561.

\bibitem {BG-GC-Zuniga}Bocardo-Gaspar Miriam, Garc\'{\i}a-Compe\'{a}n H.,
Z\'{u}\~{n}iga-Galindo W. A., Regularization of $p$-adic String Amplitudes,
and Multivariate Local Zeta Functions. arXiv:1611.03807.

\bibitem {Cassaigneetal}Cassaigne Julien, Maillot Vincent, Hauteur des
hypersurfaces et fonctions z\^{e}ta d'Igusa, J. Number Theory 83 (2000), no.
2, 226--255.

\bibitem {D0}Denef J., Report on Igusa's Local Zeta Function, S\'{e}minaire
Bourbaki 43 (1990-1991), exp. 741; Ast\'{e}risque 201-202-203 (1991), 359--386.

\bibitem {D-L}Denef J. and Loeser F., Motivic Igusa zeta functions, J. Alg.
Geom. 7 (1998), 505--537.

\bibitem {Gelfand-Shilov-1}Gel'fand \ I. M., Shilov G. E., Generalized
Functions. Vol. 1. Properties and operations. AMS Chelsea publishing, 2010.

\bibitem {Gel-Shilov}Gel'fand I. M., Shilov G. E., Generalized functions. Vol.
2. Spaces of fundamental and generalized functions. AMS Chelsea publishing, 2010.

\bibitem {Gel-Vil}Gel'fand I. M., Vilenkin N. Ya, Generalized functions. Vol.
4. Applications of harmonic analysis. AMS Chelsea publishing, 2010.

\bibitem {Hida et al}Hida Takeyuki, Kuo Hui-Hsiung, Potthoff J\"{u}rgen,
Streit, Ludwig, White noise. An infinite-dimensional calculus. Kluwer Academic
Publishers Group, 1993.

\bibitem {Igusa}Igusa J.-I., An introduction \ to the theory of local zeta
functions. AMS/IP Studies in Advanced Mathematics, 2000.

\bibitem {Igusa-SPF}Igusa J.-I., A stationary phase formula for $p$-adic
integrals and its applications, Algebraic geometry and its applications,
Springer-Verlag (1994), pp. 175--194.

\bibitem {Igusa1}Igusa J.-I., Forms of higher degree. The Narosa Publishing
House, 1978.

\bibitem {H}Hironaka H., Resolution of singularities of an algebraic variety
over a field of characteristic zero, Ann. Math. 79 (1964), 109--326.

\bibitem {Koch}Kochubei Anatoly N., Pseudo-differential equations and
stochastics over non-Archimedean fields. Marcel Dekker, 2001.

\bibitem {Lo2}Loeser F., Fonctions z\^{e}ta locales d'Igusa \`{a} plusiers
variables, int\'{e}gration dans les fibres, et discriminants, Ann. Sc. Ec.
Norm. Sup. 22 (1989), no. 3, 435--471.

\bibitem {Mustata}Musta\c{t}\u{a} Mircea, Bernstein-Sato polynomials in
positive characteristic, J. Algebra 321 (2009), no. 1, 128--151.

\bibitem {Obata}Obata Nobuaki, White noise calculus and Fock space. Lecture
Notes in Mathematics, 1577. Springer-Verlag, 1994.

\bibitem {Rubin}Rubin Boris, Riesz potentials and integral geometry in the
space of rectangular matrices, Adv. Math. 205 (2006), no. 2, 549--598.

\bibitem {Speer}Speer Eugene R., Generalized Feynman amplitudes, Annals of
Mathematics Studies, No. 62. Princeton University Press, 1969.

\bibitem {Taibleson}Taibleson M. H., Fourier analysis on local fields.
Princeton University Press, 1975.

\bibitem {VZ}Veys W. and Z\'{u}niga-Galindo W.A., Zeta functions for analytic
mappings, log-principalization of ideals and Newton polyhedra, Trans. Amer.
Math. Soc. 360 (2008), 2205--2227.

\bibitem {VZ-1}Veys W. and Z\'{u}niga-Galindo W.A., Zeta functions and
oscillatory integrals for meromorphic functions, Adv. Math. http://dx.doi.org/10.1016/j.aim.2017.02.022.

\bibitem {V-V-Z}Vladimirov V. S., Volovich I. V., Zelenov E. I., P-adic
analysis and mathematical physics. World Scientific, 1994.

\bibitem {Weil}Weil Andr\'{e}, Basic number theory. Springer-Verlag, 1967.

\bibitem {Yosida}Yosida K\^{o}saku, Functional analysis. Springer-Verlag, 1965.

\bibitem {Zuniga2016}Z\'{u}\~{n}iga-Galindo W. A., Non-Archimedean white
noise, pseudodifferential stochastic equations, and massive Euclidean fields,
J. Fourier Anal. Appl. 23, no 2 (2017) 288--323.

\bibitem {Zuniga-LNM-2016}Z\'{u}\~{n}iga-Galindo W. A., Pseudodifferential
equations over non-Archimedean spaces. Lectures Notes in Mathematics 2174,
Springer, 2016.

\bibitem {Zuniga-Nagoya-2003}Zuniga-Galindo W. A., Local zeta functions and
Newton polyhedra, Nagoya Math. J. 172 (2003), 31--58.

\bibitem {Zuniga-Padova-2003}Zuniga-Galindo W. A., Fundamental solutions of
pseudo-differential operators over $p$-adic fields, Rend. Sem. Mat. Univ.
Padova 109 (2003), 241--245.

\bibitem {Zuniga-TAMS-2001}Z\'{u}\~{n}iga-Galindo W. A., Igusa's local zeta
functions of semiquasihomogeneous polynomials, Trans. Amer. Math. Soc. 353
(2001), no. 8, 3193--3207.
\end{thebibliography}
\end{document}